\documentclass[10pt]{elsarticle}

\usepackage{url}
\usepackage{fullpage}
\usepackage{here}
\usepackage{cite}            
\usepackage{graphicx}          
\usepackage{slashbox}                               
\usepackage[T1]{fontenc}
\usepackage{grffile}
\usepackage{mathtools}

\biboptions{sort&compress}

\usepackage{amssymb}

\usepackage{amsmath}
\usepackage{graphicx}
\usepackage{amsfonts}
\usepackage{multirow}
\usepackage{enumerate}

\usepackage{color}
\usepackage{mathrsfs}
\usepackage{dsfont}

\providecommand{\blue}[1]{\color{black}{#1}\color{black}\hspace{0pt}}
\providecommand{\bblue}[1]{\color{black}{#1}\color{black}\hspace{0pt}}

\providecommand{\black}[1]{#1}

\everymath{\displaystyle}

\newtheorem{theorem}{Theorem}[section]
\newtheorem{corollary}[theorem]{Corollary}
\newtheorem{proposition}[theorem]{Proposition}

\newtheorem{define}[theorem]{Definition}
\newtheorem{definition}[theorem]{Definition}

\DeclareMathOperator*{\col}{col}
\DeclareMathOperator*{\row}{row}
\DeclareMathOperator*{\diag}{diag}
\DeclareMathOperator*{\T}{\intercal}
\DeclareMathOperator*{\He}{Sym}

\DeclareMathOperator*{\trace}{trace}
\DeclareMathOperator{\eps}{\varepsilon}

\def\T{\mathrm{{\scriptstyle T}}}

\DeclareMathOperator{\vect}{\mathrm{vec}}

\def\Tmin{T_{\textnormal{min}}}
\def\Tmax{T_{\textnormal{max}}}
\def\lmin{\lambda_{\textnormal{min}}}
\def\lmax{\lambda_{\textnormal{max}}}

\def\sp{\ensuremath{\mathrm{sp}}}

\def\dx{\mathrm{d}x}
\def\d{\mathrm{d}}
\def\dt{\mathrm{d}t}
\def\ds{\mathrm{d}s}
\def\dW{\mathrm{d}W}

\newcommand{\pd}{\ensuremath\mathbb{S}_{\succ0}}
\newcommand{\psd}{\ensuremath\mathbb{S}_{\succeq0}}

\newcommand{\sm}{\ensuremath\mathbb{S}}

\newcommand{\Rnn}{\ensuremath\mathbb{R}_{\ge0}}

\newcommand{\R}{\ensuremath\mathbb{R}}

\newcommand{\Znn}{\ensuremath\mathbb{Z}_{\ge0}}

\def\P{\mathbb{P}}
\def\R{\mathbb{R}}
\def\E{\mathbb{E}}

\def\d{\textnormal{d}}
\def\dW{\textnormal{d}W}

\def\cite{\citep}

\newenvironment{proof}{{\it Proof :~}}{\hfill$\diamondsuit$\\}

\title{Stability Analysis and Stabilization of Linear Symmetric Matrix-Valued Continuous, Discrete, and Impulsive Dynamical Systems\\ --\\ A Unified Approach for the Stability Analysis and the Stabilization of Linear Systems}

\author{Corentin Briat\footnote{email: {\tt  corentin@briat.info}; url:~{\tt http://www.briat.info}; ORCID Number: 0000-0003-1822-2683}\\Independent Researcher\\ Basel, Switzerland}

\begin{document}

\begin{frontmatter}

\begin{abstract}
Symmetric matrix-valued dynamical systems are an important class of systems that can describe important processes such as covariance/second-order moment processes, or processes on manifolds and Lie Groups. We address here the case of processes that leave the cone of positive semidefinite matrices invariant, thereby including covariance and second-order moment processes. Both the continuous-time and the discrete-time cases are first considered. In the LTV case, the obtained stability and stabilization conditions are expressed as differential and difference Lyapunov conditions which are equivalent, in the LTI case, to some spectral conditions for the generators of the processes. Convex stabilization conditions are also obtained in both the continuous-time and the discrete-time setting. It is proven that systems with constant delays are stable provided that the systems with zero-delays are stable -- which mirrors existing results for linear positive systems. The results are then extended and unified into a impulsive formulation for which similar results are obtained. The proposed framework is very general and can recover and/or extend almost all the results on the literature of linear systems related to (mean-square) exponential (uniform) stability. Several examples are discussed to illustrate this claim by deriving stability conditions for stochastic systems driven by Brownian motion and Poissonian jumps, Markov jump systems, (stochastic) switched systems, (stochastic) impulsive systems, (stochastic) sampled-data systems, and all their possible combinations.
\end{abstract}
\end{frontmatter}


\section{Introduction}

Matrix-valued dynamical systems form an important class of systems that can be used to represent several processes such as the evolution of covariances or second-order moments \citep{SkeltonIG:97a}, the dynamics of systems on manifolds such as SO(3) \citep{Abraham:88,Bullo:05} or Lie groups \citep{Jurdjevic:72}. In particular, the class of matrix-valued dynamical systems that leave the cone of positive semidefinite matrices invariant is of special interest because it can represent the dynamics of covariance \black{or seconder-order moment} matrices associated with stochastic dynamical systems. The consideration of the so-called second-order information is not new and has been considered in the past in \citep{SkeltonIG:97a}. More recent works such as \citep{You:15,Gattami:16} also consider second-order information methods for providing new ways to solve existing problems related to $H_\infty$ analysis. \bblue{Finally, this approach is also widely used in the analysis and the control of stochastic systems; see e.g. \citep{Costa:05,Antunes:10,Costa:13,Dragan:02,Dragan:04,Dragan:05,Dragan:06,Dragan:10,Dragan:13,Dragan:20,Dragan:21,Dragan:22}.}

\black{The objective of this paper is to provide a clear analysis of a very general class of matrix-valued linear dynamical systems evolving on the cone of positive semidefinite matrices as a unifying way for analyzing the stability of linear systems, including deterministic or stochastic, continuous, discrete or hybrid systems, without requiring any deep knowledge in stochastic processes or stochastic calculus. Indeed, the analysis of all those classes of systems can be brought back to the analysis of an associated matrix-valued differential or difference equations, or a combination of them in the hybrid case, which describe the evolution of the second-order moment associated with the state of the system. In fact, a very simple procedure can be applied to build those matrix-valued dynamical systems from the original system.}

\black{Another interesting point is that cone-preserving systems have been shown to admit simpler characterizations for their properties \citep{Bundfuss:09,Tanaka:13b,Chen:18}. The simplest class of cone preserving systems is that of linear positive systems \citep{Farina:00} and are known to admit linear programming conditions for their stability using linear copositive Lyapunov functions and for the characterization of the $L_1$- and $L_\infty$-gains which both coincide with the $1$- and $\infty$-norm of the DC-matrix of the system, which is obtained by evaluating the transfer function of the system at $s=0$; see e.g. \citep{Briat:11g,Briat:11h,Ebihara:11}. Analogously, the $H_\infty$-norm of such systems coincides with the $2$-norm of the DC-matrix \citep{Rantzer:12}. Those results notably had simplifying implications in the analysis of stability of power control in wireless networks in \citep{Qian:17a,Qian:17b} but also in the analysis and the scalable control of compartmental systems \citep{Rantzer:21}, or the analysis and the control of biological systems \cite{Briat:13i,Briat:15e,Briat:19:DelayedRN,Briat:20:Structural}, etc. The design of constrained or structured controllers is also convex for those systems \citep{Aitrami:07}. Linear positive systems with delays are stable provided that their zero-delay counterparts are also stable \citep{Haddad:04,AitRami:09,Briat:11h,Shen:14,Shen:15,Shen:15b,Briat:16b}. Many extensions of those results as well as new ones have been provided; see e.g. \citep{Khong:16,Rantzer:16}.  As the systems considered in this paper can be seen as matrix analogues of linear positive systems, or more generally, cone preserving systems, there is hope that such systems share similar and simpler conditions characterizing some of their properties such as stability and performance. It is notably shown in this paper, that this is the case for systems with discrete delays as also discussed in \citep{Tanaka:13b}. In fact, many of the obtained results for linear positive systems with delays can be extended to the current class systems by suitably adapting the systems themselves and considering the correct Lyapunov-Krasovskii functionals. This will not be considered here as this is not the main topic of the paper but will be addressed in future works focusing on the analysis of delay systems and on the development of a dissipativity theory for such systems.}

\black{The first step towards a possible full answer to whether matrix-valued cone preserving systems also admit simple characterization of their properties is the development of a suitable and well-rounded stability theory for them. In this paper, we use the fact that the stability analysis of linear symmetric matrix-valued continuous- and  discrete-time systems can be analyzed using linear copositive Lyapunov functions that take nonnegative values for all state values in the cone of positive semidefinite matrices; we refer those functions as linear $\psd^n$-copositive Lyapunov functions to clearly state on which cone the Lyapunov function is positive definite. In fact, those Lyapunov functions can be all characterized in terms of the interior of the dual cone, the latter coinciding with the cone of positive semidefinite matrices as this cone is self-dual. The main tools are the inner-product of matrices and the concepts of operators adjoint to those describing the dynamics of the matrix-valued dynamical systems. This general formulation allows for very simple and natural proofs for all the main results of the paper. }

\black{In the continuous-time case, both LTV and LTI systems are considered and their stability is characterized in terms of a Lyapunov differential equation and a Lyapunov equation, respectively. In the LTI case, the stability conditions are also equivalent to a condition on the eigenvalues of the operator describing the dynamics of the system. A stabilization result is also derived and takes the form of a convex differential linear matrix inequality problem. An extension of the results pertaining to a certain class of time-delay systems is also considered and extends the results in  \citep{Tanaka:13b}, while at the same time providing a new proof for it. We then show how those results can be used to retrieve some results important of the literature. The first example is that of an LTV system driven by both Brownian motions and Poissonian jumps and can be seen as a particular case of the systems considered in \cite{Antunes:09}. The second example, which extends \citep{Tanaka:13b}, pertains to the analysis of a class of diffusion processes with delays and the main result here is that this system is exponentially stable if and only if the system with zero delays is exponentially stable. The third example addresses the stability of linear time-varying continuous-time Markov jump systems considered e.g. in \citep{Kushner:67,Khasminskii:12,Costa:13} in the LTI case. We can recover those results in only few lines of calculations using the results of the paper. The last example is that of sampled-data system with Poissonian sampling that addresses, in a slightly different way, the problem considered in \citep{Verriest:09c}.}

\begin{figure}[H]
  \centering
  \includegraphics[width=0.99\textwidth]{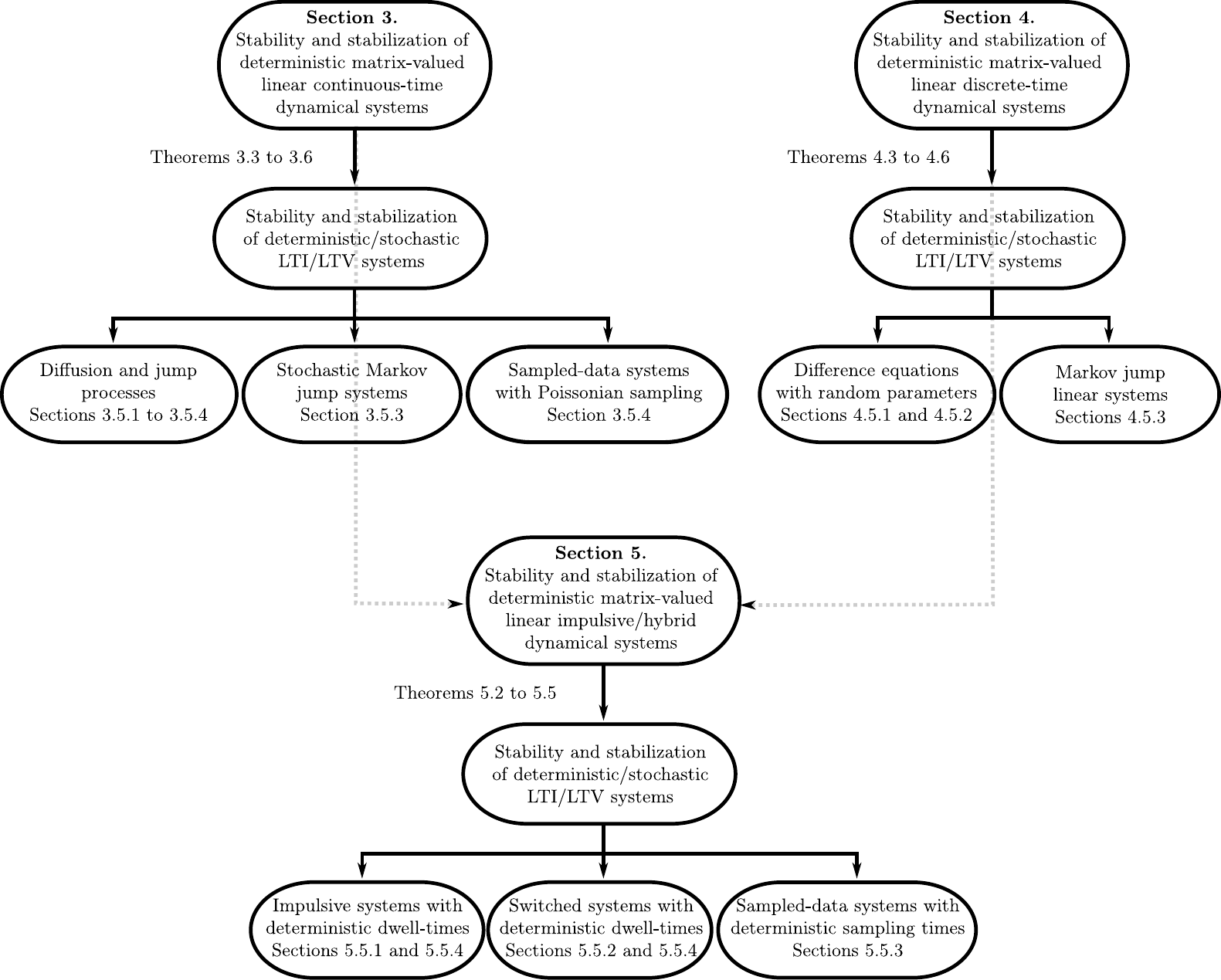}\\
  \caption{Structure of the paper and types of systems that can be considered within the considered framework.}\label{structure}
\end{figure}

\black{We obtain then analogous results in the discrete-time case where both the cases of LTV and LTI systems are addressed. In this context, the stability conditions are rather expressed in terms of Lyapunov-Stein difference equations and Lyapunov-Stein equations, respectively. A spectral condition is also obtained in the LTI case. Stabilization conditions are formulated in terms of a difference linear matrix inequality. A result on delay systems is also obtained, with the same conclusions as in the continuous-time case. However, this result does not seem to have been previously obtained in the literature. Again, some examples are considered in order to demonstrate the usefulness of the approach. The first example is that of stochastic LTV systems involving two different classes of stochastically varying parameters. The first class is that of identically distributed independent parameters with zero mean and unit variance whereas the second one is that of independent non-identically distributed Bernoulli random variables. The second example pertains to the analysis of a class of linear stochastic discrete-time system with delays, whereas the last one considers time-varying Markov jump linear systems for which we recover some of the results in \citep{Costa:05}.}

\black{The two previously considered frameworks are then combined to give rise to matrix-valued impulsive systems which are described as coupled matrix-valued differential and difference equations. As a result, the stability conditions are expressed in terms of coupled differential-/difference equations in both the LTV and LTI cases. Following the ideas in \citep{Goebel:12}, those stability conditions are relaxed giving rise to persistent flowing and persistent jumping stability conditions which only assume a strict decrease of the Lyapunov function along the flow and at the jumps, respectively. Stabilization conditions are also obtained and stated in terms of coupled convex differential and difference linear matrix inequalities. As an illustration of the main result, the first example addresses a class of stochastic LTV systems driven by Brownian motions and subject to deterministic impulses and it is shown that we can easily recover and extend the results obtained in \citep{Briat:15i}. The second example pertains to a class of stochastic LTV switched systems that generalizes the classes considered in \citep{Shaked:14,Briat:14f} and for which we retrieve those results as special cases of those obtained in the current framework. The next example considers the case of stochastic LTV sampled-data systems which we reformulate as a stochastic hybrid system. In the case when the sampling times are deterministic, we obtain a generalization of the results in \citep{Briat:13d,Briat:15i,Goebel:12}. Finally, the last example addresses the analysis of a certain class of stochastic LTV impulsive systems with stochastic impulses and switching first considered in \citep{Souza:21} and for which we can recover and generalize the existing results to the LTV case.}\\

\noindent\textbf{Outline.} Preliminary definitions are given in Section \ref{sec:Preliminaries}. Continuous-time matrix-valued differential equations are considered in Section \ref{sec:CT} whereas their discrete-time counterparts in Section \ref{sec:DT}. The results are unified in a hybrid formulation in Section \ref{sec:Hybrid}. Application examples are considered in the related sections.\\

\noindent\textbf{Notations.} The set of positive and nonnegative integers are denoted by $\mathbb{Z}_{>0}$ and $\mathbb{Z}_{\ge0}$, respectively. The set of positive and nonnegative real numbers are denoted by $\mathbb{R}_{>0}$ and $\mathbb{R}_{\ge0}$, respectively. Similarly, the cone of \blue{symmetric} positive definite and \blue{symmetric} positive semidefinite matrices of dimension $n$ are denoted by $\pd^n$ and $\psd^n$, respectively.\blue{ The set of symmetric matrices of dimension $n$ is simply denoted by $\sm^n$.} The left open right-half plane of the complex plane is denoted by $\mathbb{C}_{<0}$. The natural basis for the Euclidian space is denoted by $\{e_1,\ldots\}$. For a square real matrix $A$, we denote $\He[A]:=A+A^{\T}$. For a matrix $A$ with columns $\{a_1,\ldots,a_n\}$, the vectorization operator $\vect(\cdot)$ stacks the columns of a matrix on the top of each other; i.e. $\vect(A)=\begin{bmatrix}   a_1^{\T} & \ldots & a_n^{\T} \end{bmatrix}^{\T}$. 
For a square matrix $A$, the minimum and maximum eigenvalues of $A$ are denoted by $\lmin(A)$ and $\lmax(A)$, respectively. The Hadamard product is denoted by $\odot$ whereas the Kronecker product and sum are denoted by $\otimes$ and $\oplus$, respectively. \blue{Finally, $\star$ stands for symmetric entries in symmetric matrices.}

\section{Preliminaries}\label{sec:Preliminaries}

\black{We state in this section the essential definitions and results that will be key for deriving the main results of the paper.

 \begin{definition}[Nuclear norm \citep{Bhatia:97}]
    Let $A\in\mathbb{R}^{n\times m}$, then its nuclear-norm $||A||_*$ is defined as
    \begin{equation}
        ||A||_*:=\sum_{i=1}^{\min\{m,n\}}\sigma_i(A)
      \end{equation}
      where $\sigma_i(A)$ is the $i$-th singular values of $A$. When $A\in\psd^n$, then we have that
      \begin{equation}
       ||A||_*=\trace(A).
      \end{equation}
  \end{definition}

The following inner-product, also called Frobenius inner product, extends the concept of scalar product to matrices:
  \begin{definition}[Inner product on $\mathbb{R}^{n\times n}$]
    The inner product $\langle\cdot,\cdot\rangle$ on $\mathbb{R}^{n\times n}$ is defined as
    \begin{equation}
      \langle A,B\rangle:=\trace AB^{\T}
    \end{equation}
    for any $A,B\in\mathbb{R}^{n\times n}$.
  \end{definition}

The following result will prove to be instrumental in proving the main stability results of the paper \blue{and is proven here for the sake of completeness}:
\begin{proposition}\label{prop:PQ}
  Let us consider matrices $P\in\mathbb{S}^n_{\succ0}$, $Q\in\mathbb{S}^n_{\succeq0}$ and $R\in\mathbb{S}^n$. Then, the following statement holds:
  \begin{enumerate}[(a)]
    \item\label{st:norm:2}  Let $P\in\mathbb{S}^n_{\succ0}$ and $Q\in\mathbb{S}^n_{\succeq0}$, then we have that $\langle P,Q\rangle=0$ if and only if $Q=0$.
    \item\label{st:norm:3} Let  $R\in\mathbb{S}^n$, then we have that $\langle R,Q\rangle<0$ (resp. $\langle R,Q\rangle\le0$)  for all $Q\succeq0$, $Q\ne0$, if and only if $R\prec0$ (resp. $R\preceq0$).
    \item\label{st:norm:4}  Let $P\in\mathbb{S}^n_{\succ0}$, then we have that $\lmin(P)||Q||_*\le\langle P,Q\rangle\le\lmax(P)||Q||_*$ for all $Q\succeq0$.
  \end{enumerate}
\end{proposition}

\begin{proof}
Throughout this proof, we denote by $\{v_1,\ldots,v_n\}$ the set of orthonormal eigenvectors of the matrix $P$ associated with the eigenvalues
$\{\lambda_1,\ldots,\lambda_n\}$. Observe also that the matrix $V=[v_1\ \ldots\ v_n]$ forms an orthonormal basis of $\mathbb{R}^n$.\\


  \noindent\textbf{Proof of statement \eqref{st:norm:2}.} Clearly if $Q=0$, then $\langle P,Q\rangle=0$. Assume on the other hand that $\langle P,Q\rangle=0$. This is equivalent to saying that
  \begin{equation}
    \sum_i\lambda_iv_i^\T Qv_i=0
  \end{equation}
  and we necessarily have that $v_i^\T Qv_i=0$ for all $i=1,\ldots,n$ since $\lambda_i>0$ for all $i=1,\ldots,n$. Since $V$ is a basis of $\mathbb{R}^n$,
  then $V^\T QV$ is similar to $Q$ and the diagonal elements of $V^\T QV$  are all zero. Since the trace of a matrix is the sum of its eigenvalues and since
  $Q$ is positive semidefinite, then this implies that $Q$ must be the zero matrix. This proves statement \eqref{st:norm:2}.\\

  \noindent\textbf{Proof of statement \eqref{st:norm:3}.} 
   Assume first that $R\prec0$. Then,
  we have that $\langle R,Q\rangle=\textstyle\sum_{i=1}^n\mu_iw_i^\T Qw_i$ where $\mu_i$ and $w_i$ are the eigenvalues and the eigenvectors of the matrix
  $R$. This implies that $\langle R,Q\rangle<0$ for all $Q\succeq0$, $Q\ne0$.

  We show now the reverse implication. Assume that $\langle R, Q\rangle<0$ for all $Q\succeq0$, $Q\ne0$. Then, this must hold for all $Q=Q_i=w_iw_i^\T
  \succeq0$, $i=1,\ldots,n$. This implies that we must have that $\langle R, Q_i\rangle=\mu_i<0$. This proves the necessity and the statement \eqref{st:norm:3}.\\

  \noindent\textbf{Proof of statement \eqref{st:norm:4}.} Since $Q$ is symmetric, then it can be written as $Q=\textstyle\sum_{i=1}^n\theta_i z_iz_i^{\T}$
  where $\theta_i $ and $z_i$ are its eigenvalues and eigenvectors. Then,
\begin{equation}
  \begin{array}{rcl}
    \langle P,Q\rangle &=&  \sum_{i=1}^n\theta_i\langle P,z_iz_i^\T\rangle=\sum_{i=1}^n\theta_iz_i^\T Pz_i\\
                    &\le&\lambda_{\mathrm{max}}(P)\sum_{i=1}^n\theta_i\\
                    &=&\lambda_{\mathrm{max}}(P)||Q||_*.
  \end{array}
\end{equation}
The lower bound is proven analogously. This proves statement \eqref{st:norm:4}.
\end{proof}

We will finally need the following definitions:
\begin{define}[$\psd^n$-copositive definiteness]
  A continuous function $V:\mathbb{R}_{\ge0}\times\psd^n\mapsto\mathbb{R}$ is said to be uniformly $\psd^n$-copositive definite if and only if
  \begin{enumerate}[(a)]
    \item $V(t,0)=0$ for all $t\ge0$, and
    \item  there exists a continuous, strictly increasing function  $\alpha:\mathbb{R}_{\ge0}\mapsto\mathbb{R}_{\ge0}$, $\alpha(0)=0$, such that $V(t,X)\ge\alpha(||X||)$ for all $(t,X)\in\mathbb{R}_{\ge0}\times\psd^n$, \blue{and where $||\cdot||$ denotes any matrix norm; e.g. the nuclear norm.}
  \end{enumerate}
  Moreover, the function $V$ is said to be radially unbounded if $\alpha(\cdot)$ is such that $\alpha(s)\to\infty$ as $s\to\infty$.
\end{define}
%

\begin{define}[Decrescent function \citep{Murray:94}]
  A continuous function $V:\mathbb{R}_{\ge0}\times\psd^n\mapsto\mathbb{R}_{\ge0}$ is said to be decrescent if there exists a continuous, strictly increasing function  $\alpha:\mathbb{R}_{\ge0}\mapsto\mathbb{R}_{\ge0}$ with $\alpha(0)=0$ and $\alpha(s)\to\infty$ as $s\to\infty$ such that $V(t,X)\le\alpha(||X||)$ for all $(t,X)\in \mathbb{R}_{\ge0}\times\psd^n$, \blue{and where $||\cdot||$ denotes any matrix norm; e.g. the nuclear norm.}
\end{define}

}

\section{Linear Matrix-Valued Symmetric Continuous-Time Systems}\label{sec:CT}

\blue{This section is devoted to the definition of linear matrix-valued symmetric continuous-time systems that evolves on the cone of positive semidefinite matrices together with their associated generators and state-transition operator. Using those definitions a necessary and sufficient condition for their uniform exponential stability is obtained in Theorem \ref{th:stabMatCT_LTV} in terms of a matrix-differential equation or a matrix-differential inequality. This result is then specialized to the LTI case in Theorem \ref{th:stabMatCT_LTI} where additional spectral conditions are also obtained. Convex stabilizability and stabilization conditions are also obtained in Theorem \ref{th:CT:stabz} and take the form of differential linear matrix inequalities. For completeness, the case of systems with delays is considered and a necessary and sufficient stability condition, generalizing existing ones, is also obtained in Theorem \ref{th:CT:delay}. The simplicity and the versatility of the approach as a general tool for the analysis of (stochastic) continuous-time systems is illustrated through several examples, namely on the analysis of linear time-varying stochastic systems subject to Brownian motions and Poissonian jumps, LTI time-delay systems subject to Brownian motions, time-varying Markov jump linear systems, and time-varying sampled-data systems with Poissonian sampling.}

\subsection{Preliminaries}

\black{Let us consider here the following class of matrix-valued continuous-time dynamical systems
\begin{equation}\label{eq:matdiffeqCT_LTV}
\begin{array}{rcl}
  \dot{X}(t)&=&\He[A_0(t) X(t)]+\sum_{i=1}^NA_i(t) X(t)A_i(t)^\T+\mu(t) X(t),t\ge t_0\\
  X(t_0)&=&X_0\in\psd^n
\end{array}
\end{equation}
where $A_i:\mathbb{R}_{\ge0}\mapsto\mathbb{R}^{n\times n}$, $i=0,\ldots,N$, and $\mu:\mathbb{R}_{\ge0}\mapsto\mathbb{R}$ are piecewise-continuous and bounded. It is immediate to verify that the right-hand side is uniformly Lipschitz in $X$ which implies that there exists a unique global solution which can be written as
\begin{equation}
  X(t)=S(t,s)X(s), t\ge s\ge0
\end{equation}
where $S(\cdot,\cdot)$ is the state transition operator verifying
\begin{equation}
  \dfrac{\partial}{\partial t}S(t,s)X=\mathcal{C}_t(S(t,s)X)\blue{\textnormal{ for all }}X\in\psd^n
\end{equation}
where the generator $\mathcal{C}_t$ is given by
\begin{equation}\label{eq:operatorLt}
  \begin{array}{rccl}
    \mathcal{C}_t:&\sm^n&\mapsto&\sm^n\\
    & X&\mapsto & A_0(t)X+XA_0(t)^\T+\sum_{i=1}^NA_i(t) XA_i(t)^\T+\mu(t) X.
  \end{array}
\end{equation}
The adjoint operator of $ \mathcal{C}_t$, denoted by $ \mathcal{C}_t^*$, is given by
\begin{equation}\label{eq:operatorLtadjoint}
  \begin{array}{rccl}
    \mathcal{C}_t^*:&\sm^n&\mapsto&\sm^n\\
    & X&\mapsto & A_0(t)^\T X+XA_0(t)+\sum_{i=1}^NA_i(t)^\T XA_i(t)+\mu(t) X.
  \end{array}
\end{equation}


%
The following result shows that the system \eqref{eq:matdiffeqCT_LTV} leaves the cone $\mathbb{S}^n_{\succeq0}$ invariant:
\begin{proposition}\label{prop:positivityCT}
  The matrix-valued differential equation \eqref{eq:matdiffeqCT_LTV} leaves the cone of positive semidefinite matrices invariant, that is, for any
  $X_0\succeq0$, we have that $X(t,t_0,X_0)\succeq0$ for all $t\ge0$.
\end{proposition}

\begin{proof}
  The solution to \eqref{eq:matdiffeqCT_LTV} can be shown to satisfy the expression
  \begin{equation}
    X(t)=\Phi(t,t_0)X_0\Phi(t,t_0)^{\T}+\sum_{i=1}^N\int_{t_0}^t\Phi(t,s)A_i(s) X(s)A_i(s)^\T\Phi(t,s)^{\T}\ds
  \end{equation}
  where $\Phi(t,s)$ is the state-transition matrix associated with the system $\textstyle\dot{x}(t)=(A_0(t)+\frac{1}{2}\mu(t)I_n)x(t)$. The result can be obtained by the successive approximations procedure; see e.g. Proposition 2.5 from \cite{Dragan:04}.
  \end{proof}

\begin{definition} \label{def:CTsystemMat}
The zero solution of the system \eqref{eq:matdiffeqCT_LTV} is said to be \textbf{globally uniformly exponentially stable with rate $\alpha$} if there exist some scalars $\alpha>0$ and $\beta\ge1$ such that
   \begin{equation}
                  ||X(t)||\le\beta e^{-\alpha (t-t_0)}||X_0||
   \end{equation}
for all $t\ge t_0$, all $t_0\ge0$, and all $X_0\in\psd^n$.
\end{definition}}

\subsection{Stability Analysis}

With the previous definitions and results in mind, we are now able to state the following result\footnote{\bblue{It was brought to the attention of the author after publication of the paper that this result can be seen as a particular case of Theorem 2.4.2 in \cite{Dragan:13}.}} characterizing the uniform exponential stability of the system \eqref{eq:matdiffeqCT_LTV}
\black{\begin{theorem}\label{th:stabMatCT_LTV}
  The following statements are equivalent:
  \begin{enumerate}[(a)]
  \item\label{st:stabMatCT_LTV1} The system \eqref{eq:matdiffeqCT_LTV} is uniformly exponentially stable.
    %
    %
    \item\label{st:stabMatCT_LTV3}  There exist a differentiable matrix-valued function $P:\mathbb{R}_{\ge0}\mapsto\mathbb{S}^n_{\succ0}$ and a
        continuous matrix-valued function $Q:\mathbb{R}_{\ge0}\mapsto\mathbb{S}^n_{\succ0}$ such that
    \begin{equation}
    \begin{array}{rcccl}
      \alpha_1I&\preceq& P(t)&\preceq &\alpha_2I,\\
      \beta_1I&\preceq &Q(t)&\preceq &\beta_2I,
    \end{array}
    \end{equation}
    and such that \textbf{Lyapunov differential equation}
    \begin{equation}
      \dot{P}(t)+\mathcal{C}_t^*(P(t))+Q(t)=0
  \end{equation}
  hold for all $t\ge0$ and for some scalars $\alpha_1,\alpha_2,\beta_1,\beta_2>0$.
    \item\label{st:stabMatCT_LTV4}  There exists a differentiable matrix-valued function $P:\mathbb{R}_{\ge0}\mapsto\mathbb{S}^n_{\succ0}$ such that
        the
    \begin{equation}
      \alpha_1I\preceq P(t) \preceq \alpha_2I,
    \end{equation}
    and such that \textbf{Lyapunov differential inequality}
    \begin{equation}
    \dot{P}(t)+\mathcal{C}_t^*(P(t))\preceq-\alpha_3 I
  \end{equation}
  hold for all $t\ge0$ \blue{and for some scalars $\alpha_1,\alpha_2,\alpha_3>0$}.
  \end{enumerate}
\end{theorem}

\begin{proof}
It is immediate to see that the statements \eqref{st:stabMatCT_LTV3} and \eqref{st:stabMatCT_LTV4} are equivalent.\\

\noindent \textbf{Proof that statement \eqref{st:stabMatCT_LTV3} implies statement \eqref{st:stabMatCT_LTV1}.} To this aim, define the function $V(t,X)=\langle P(t),X\rangle$ where $P(\cdot)$ satisfies the conditions of statement (b). Since $P(t)$ is positive definite, uniformly bounded and uniformly bounded away from 0, then, from Proposition \ref{prop:PQ}, we have that $V$ is $\psd^n$-copositive definite, radially unbounded and decrescent. Evaluating the derivative of the function $V$ along the trajectories of the system \eqref{eq:matdiffeqCT_LTV} yields
\begin{equation}
  \begin{array}{rcl}
    \dot{V}(t,X(t))&=&\langle \dot{P}(t),X(t)\rangle+\langle P(t),\mathcal{C}_t(X(t))\rangle\\
                                &=&\langle \dot{P}(t),X(t)\rangle+\langle \mathcal{C}^*_t(P(t)),X(t)\rangle\\
                                &=&-\langle Q(t),X(t)\rangle\\
                                &\le&-\dfrac{\beta_1}{\alpha_2}\langle P(t),X(t)\rangle=-\dfrac{\beta_1}{\alpha_2}V(t,X(t)).
  \end{array}
\end{equation}
Therefore, $V$ is a global, uniform and exponential Lyapunov function for the system \eqref{eq:matdiffeqCT_LTV}, which proves that the system \eqref{eq:matdiffeqCT_LTV} is globally uniformly exponentially stable with rate $\beta_1/\alpha_2$.\\

\noindent\textbf{Proof that statement \eqref{st:stabMatCT_LTV1}  implies statement \eqref{st:stabMatCT_LTV3}.} To this aim, assume that the system \eqref{eq:matdiffeqCT_LTV} is uniformly exponentially stable, then there exist $M,\theta>0$ such that
\begin{equation}
  ||S(s,t)X||_*\le Me^{-\theta(s-t)}||X||_*,s\ge t\ge0
\end{equation}
for all $X\in\psd^n$. We then define $\bar{P}:\mathbb{R}_{\ge0}\mapsto\pd^n$ as
  \begin{equation}\label{eq:kdsakdks;ld456456465467468}
    \langle \bar{P}(t),X(t)\rangle=\int_t^\infty\langle Q(s),X(s)\rangle\d s
  \end{equation}
  where $Q:\mathbb{R}_{\ge0}\mapsto\pd^n$ and such that $\beta_1I\preceq Q(t)\preceq \beta_2$ for some $\beta_1,\beta_2>0$ and for all $t\ge0$.  Note that $\bar{P}$ is uniquely defined since the expression is linear and it must hold for all $X(t)\succeq0$ and all $t\ge0$.  We need to prove first that there exist positive constant $\alpha_1,\alpha_2$ such that $\alpha_1I\preceq \bar P(t)\preceq\alpha_2I$ for all $k\ge0$.

Clearly, we have that
\begin{equation}
\begin{array}{rcl}
 \langle \bar{P}(t),X(t)\rangle   &=& \lim_{\tau\to\infty}\int_t^\tau \langle Q(s),S(s,t)X(t)\rangle\ds\\
                &\le&  \lim_{\tau\to\infty}\beta_2||X(t)||_*\int_t^\tau e^{-\theta(s-t)}\ds\\
                &\le&  \lim_{\tau\to\infty}\dfrac{\beta_2M(1-e^{-\theta\tau})}{\theta}||X(t)||_*\\
                    &=&\dfrac{\beta_2M}{\theta}||X(t)||_*=:\alpha_2||X(t)||_*
\end{array}
\end{equation}
which implies that $\bar P(t)\preceq\alpha_2I$. We need to prove now that there exists an $\alpha_1>0$ such that $\bar P(t)\succeq\alpha_1I$. To this aim, assume that $Q$ is differentiable with uniformly bounded derivative (which is a weak assumption as $Q$ can always be chosen to be a constant). We then obtain
\begin{equation}
\begin{array}{rcl}
  \dfrac{\d}{\d s}\langle Q(s),X(s)\rangle&=&\langle \dot{Q}(s),X(s)\rangle+\langle Q(s),\mathcal{C}_s(X(s))\rangle\\
  &=& \langle \dot{Q}(s),X(s)\rangle+\langle \mathcal{C}_s^*(Q(s)),X(s)\rangle\\
  &=& \langle \dot{Q}(s)+\mathcal{C}_s^*(Q(s)),X(s)\rangle\\
  &=&\left\langle\dot{Q}(s)+\He[Q(s)A_0(s)]+\sum_{i=1}^NA_i(s)^{\T}Q(s)A_i(s),X(s)\right\rangle.
\end{array}
\end{equation}
Since both $Q$ and $\dot Q$ are uniformly bounded, then there exists a sufficiently large $c>0$ such that
\begin{equation}
\dfrac{\d}{\d s}\langle Q(s),X(s)\rangle\ge-c\cdot\langle Q(s),X(s)\rangle
\end{equation}
for all $X(s)\in\psd^n$. Integrating both sides from $t$ to $\infty$ yields
\begin{equation}
  \int_t^\infty\dfrac{\d}{\d s}\langle Q(s),X(s)\rangle\d s\ge-c\cdot\langle \bar{P}(t),X(t)\rangle.
\end{equation}
or, equivalently,
\begin{equation}
  -\langle Q(t),X(t)\rangle\ge-c\cdot\langle \bar{P}(t),X(t)\rangle
\end{equation}
where we have used that $X$ is uniformly exponentially stable. Reorganizing the above expression, we get that
\begin{equation}
  \langle Q(t)-c\bar{P}(t),X(t)\rangle\le0.
\end{equation}
As this inequality holds for all $X(t)\in\psd^n$, this implies that $\textstyle\bar{P}(t)\succeq Q(t)/c\succeq \frac{\beta_1}{c}I=:\alpha_1I$. So, we have proven that there exist positive constant $\alpha_1,\alpha_2$ such that $\alpha_1I\preceq \bar P(t)\preceq\alpha_2I$ for all $k\ge0$.\\

  Computing now the time-derivative of \eqref{eq:kdsakdks;ld456456465467468} yields
  \begin{equation}
    \langle\dot{\bar{P}}(t),X(t)\rangle+\langle\bar{P}(t),\mathcal{C}_t(X(t))\rangle=-\langle Q(t),X(t)\rangle
  \end{equation}
  which is equivalent to
  \begin{equation}
    \langle\dot{\bar{P}}(t)+\mathcal{C}_t^*(\bar P(t))+Q(t),X(t)\rangle=0.
  \end{equation}
  As this is true for all $X(t)\succeq0$, then this implies that the conditions of statement \eqref{st:stabMatCT_LTV3} hold with $P=\bar{P}$ and the same   $Q$. This proves the desired result.
\end{proof}}

\black{The above result demonstrates that the stability of systems of the form \eqref{eq:matdiffeqCT_LTV} can be proven using a specific class of linear $\psd^n$ -copositive Lyapunov functions in a way that is similar to that of the analysis of linear positive systems using linear $\R_{\ge0}^n$-copositive Lyapunov functions. The proposed approach leads to stability conditions taking the form differential Lyapunov equations of inequalities.

In the case of LTI systems, however, we can also connect the stability properties of the system \eqref{eq:matdiffeqCT_LTV} to the spectral properties of a matrix and those of a linear operator as shown below:
\begin{theorem}[LTI case]\label{th:stabMatCT_LTI}
  Assume that the system \eqref{eq:matdiffeqCT_LTV} is time-invariant. Then, the following statements are equivalent:
  \begin{enumerate}[(a)]
    \item The system \eqref{eq:matdiffeqCT_LTV} is exponentially stable.
    \item The conditions of Theorem \ref{th:stabMatCT_LTV}, \eqref{st:stabMatCT_LTV3} and \eqref{st:stabMatCT_LTV4} hold with constant matrix-valued functions $P(t)\equiv P\in\pd^n$ and $Q(t)\equiv Q\in\pd^n$.
  \item The matrix
  \begin{equation}
    \mathcal{M}_{\mathcal{C}}:=F^{\T}\left(I_n\otimes A_0+A_0\otimes I_n+\sum_{i=1}^NA_i\otimes A_i+\alpha I_{n^2}\right)F
  \end{equation}
  is Hurwitz stable where $F\in\mathbb{R}^{n^2\times n(n+1)/2}$, $F^{\T}F=I$ is such that
  \begin{equation}
    \vect(X)=F\overline\vect(X)
  \end{equation}
  where $\overline\vect(X)$ is given by
  \blue{\begin{equation}
    \left[\begin{array}{ccccccccccccccccccccc}
      X_{11} & \sqrt{2}X_{21} & \ldots & \sqrt{2}X_{n1} & \vline & X_{22} & \sqrt{2}X_{32} & \ldots & \sqrt{2}X_{n2} & \vline & \ldots & \vline & X_{nn}.
    \end{array}\right]
  \end{equation}}
  \item We have that $\sp(\mathcal{C})\subset\mathbb{C}_{<0}$ where
  \begin{equation}
        \sp(\mathcal{C}):=\left\{\lambda\in\mathbb{C}: \mathcal{C}(X)=\lambda X\textnormal{ for some }X\in\mathbb{S}^{n},X\ne0\right\}
  \end{equation}
    and $\mathcal{C}$ is defined as the time-invariant version of the operator \eqref{eq:operatorLt}.
  \end{enumerate}
\end{theorem}
\begin{proof}
The proof of the equivalence between the two first statements follow from Theorem \ref{th:stabMatCT_LTV} specialized to the LTI case. \\

\noindent\textbf{Proof that statement (a) is equivalent to statement (d).} This follows from the fact that the LTI version of the system \eqref{eq:matdiffeqCT_LTV} is globally exponentially stable if and only if the eigenvalues associated with its dynamics have negative real part. Indeed, the solution to that LTI system is given by
  \begin{equation}
    X(t)=\sum_{i=1}^{n(n+1)/2}e^{\lambda_it}\langle X_0,X_i\rangle
  \end{equation}
  where $(\lambda_i,X_i)$ denotes the $i$-th pair of eigenvalues and eigenmatrices.\\

\noindent\textbf{Proof that statement (c) is equivalent to statement (d).} Using Kronecker calculus, Statement (d) is equivalent to saying that
\begin{equation}
  \left(I_n\otimes A_0+A_0\otimes I_n+\sum_{i=1}^NA_i\otimes A_i+\alpha I_{n^2}-\lambda I_{n^2}\right)\vect(X)=0,\ X\ne0,X\in\mathbb{S}^n
\end{equation}
implies $\Re(\lambda)<0$. The eigenvalues depend on the domain of the operator which is not the full space $\mathbb{R}^{n^2}$ here but the subspace associated with symmetric matrices $X$. This can be done by restricting the matrix  above to that subspace by projection. This can be achieved using the matrix $F$ as
\begin{equation}
  \left(\mathcal{M}_{\mathcal{C}}-\lambda I_{\frac{n(n+1)}{2}}\right)\overline\vect(X)=0,
\end{equation}
which is equivalent to statement (c).
\end{proof}}

The above result sheds some light and draws a parallel between standard results for linear systems theory where stability properties are connected to spectral properties of some operators/matrices. \blue{The advantage of such a result lies in its reduced computational complexity compared to a solution based on the (reduced) vectorization of the state of the matrix-valued dynamical system, which is of dimension $n(n+1)/2$. In such a case, a Lyapunov stability condition would require finding a Lyapunov function consisting of $$\dfrac{n(n+1)}{4}\left(\dfrac{n(n+1)}{2}+1\right)=O(n^4)$$ decision variables. On the other hand, the Lyapunov inequality in Theorem \ref{th:stabMatCT_LTV}, \eqref{st:stabMatCT_LTV4} reduced to the LTI case only involves $$\dfrac{n(n+1)}{2}=O(n^2)$$ decision variables, which is clearly much less. The underlying reason is that when considering the matrix-valued dynamical system, the positive semidefiniteness of the state of the dynamical system can be easily considered unlike in the vectorized case where the positive semidefiniteness does not transfer well to this representation structure. In fact, we would need to consider extra nonlinear semialgebraic constraints to characterize the manifold on which the state evolves. This is much more complicated than the proposed approach, which is very natural to consider for solving the current problem.}

\subsection{Stabilization}

\blue{Let us consider the perturbation of the operator $\mathcal{C}_t$ given by
\begin{equation}\label{eq:operatorLtB1}
  \begin{array}{rccl}
    \tilde{\mathcal{C}}_t:&\sm^n&\mapsto &\sm^n\\
    & X&\mapsto & (A_0(t)+B_0(t)K_0(t))X+X(A_0(t)+B_0(t)K_0(t))^\T\\
    &&&\qquad\qquad+\sum_{i=1}^N(A_i(t)+B_i(t)K_i(t))X(A_i(t)+B_i(t)K_i(t))^\T+\mu(t) X.
  \end{array}
\end{equation}
where the matrix-valued functions $A_i:\mathbb{R}_{\ge0}\mapsto\mathbb{R}^{n\times n}$, $B_i:\mathbb{R}_{\ge0}\mapsto\mathbb{R}^{n\times m_i}$, $K_i:\mathbb{R}_{\ge0}\mapsto\mathbb{R}^{m_i\times n}$, $i=0,\ldots,N$, are assumed to be piecewise continuous and bounded. The objective is to formulate conditions for the existence of some $K_i$'s such that the above operator generates an uniformly exponentially stable state transition operator; i.e. $\dot{X}(t)=\tilde{\mathcal{C}}_t(X(t))$ is globally uniformly exponentially stable. This is formulated in the following result:

\begin{theorem}\label{th:CT:stabz}
  The following statements are equivalent:
  \begin{enumerate}[(a)]
    \item There exist matrix-valued functions $K_i:\mathbb{R}_{\ge0}\mapsto\mathbb{R}^{m_i\times n}$, $i=0,\ldots,$, such that $\dot{X}(t)=\tilde{\mathcal{C}}_t(X(t))$ is uniformly exponentially stable.
    \item There exists a differentiable matrix-valued function $Q:\mathbb{R}_{\ge0}\mapsto\pd^n$, $\alpha_1I\preceq Q(t)\preceq\alpha_2I$, such that the differential matrix inequality
    \begin{equation}
        \diag_{i=0}^N[\mathscr{N}_i(t)^{\T}]
      \begin{bmatrix}
            \dot{P}(t) + A_0(t)Q(t)+Q(t)A_0^{\T}+\mu(t)Q(t) & \row_{i=1}^N[(A_i(t)]Q(t))^{\T}\\
             \star  & -I_N\otimes Q(t)
        \end{bmatrix}\diag_{i=0}^N[\mathscr{N}_i(t)]\preceq -\alpha_3 I
    \end{equation}
    holds for some $\alpha_1,\alpha_2,\alpha_3>0$ and where $\mathscr{N}_i$ is such that $\mathscr{N}_i(t)^{\T}$ is a basis for the left null-space of $B_i(t)$; i.e. $\mathscr{N}_i(t)^{\T}B_i(t)=0$ and $\mathscr{N}_i(t)$ is of maximal rank.
    \item There exist a differentiable matrix-valued function $Q:\mathbb{R}_{\ge0}\mapsto\pd^n$, $\alpha_1I\preceq Q(t)\preceq\alpha_2I$, and piecewise continuous matrix-valued functions $U_i:\mathbb{R}_{\ge0}\mapsto\mathbb{R}^{m_i\times n}$ such that the differential matrix inequality
    \begin{equation}
\begin{bmatrix}
            -\dot{Q}(t) + \He[A_0(t)Q(t)+B_0(t)U_0(t)]+\mu(t)Q(t) & \row_{i=1}^N[(A_i(t)Q(t)+B_i(t)U_i(t))^{\T}]\\
             \star  & -I_N\otimes Q(t)
        \end{bmatrix}\preceq -\alpha_3 I
    \end{equation}
    holds for some $\alpha_1,\alpha_2,\alpha_3>0$. Moreover, suitable $K_i$'s are given by $K_i(t)=U_i(t)Q(t)^{-1}$, $i=0,\ldots,N$.
  \end{enumerate}
\end{theorem}
\begin{proof}
  The equivalence between the first and the third statements follows from substituting the matrices of the operator \eqref{eq:operatorLtB1} into the condition of the statement \eqref{st:stabMatCT_LTV3} of Theorem \eqref{th:stabMatCT_LTV}. Performing then a congruence transformation with respect to $Q(t):=P(t)^{-1}$, making the changes of variables  $U_i(t)=K_i(t)Q(t)$, and using a Schur complement leads to the result. The equivalence between the two last statements follows from the fact that the differential matrix inequality in the third statement can be rewritten as
  \begin{equation}
    \begin{bmatrix}
            \dot{P}(t) + \He[A_0(t)Q(t)]+\mu(t)Q(t) & \row_{i=1}^N[(A_i(t)Q(t))^{\T}]\\
             \star  & -I_N\otimes Q(t)
        \end{bmatrix}+\He\left(\diag_{i=0}^N B_i(t)\col_{i=0}^N(U_i(t))\begin{bmatrix}
          I_n & 0_{n\times nN}
        \end{bmatrix}\right)\preceq -\alpha_3 I.
  \end{equation}
  Invoking the Projection Lemma \cite{Gahinet:94a} yields the condition in the second statement.
\end{proof}}

\subsection{Systems with delays}

\black{Interestingly, it is possible to extend some of the above results to the case of systems with delays, leading to an extension of the results in \citep{Tanaka:13b}. The LTI time-delay system version of the system \eqref{eq:matdiffeqCT_LTV} is given by
\begin{equation}\label{eq:matdiffeqCT_LTV:delay}
\begin{array}{rcl}
  \dot{X}(t)&=&\He[A_0 X(t)]+\sum_{i=1}^NA_i X(t)A_i^\T+\mu X(t)+\sum_{i=1}^NB_i X_t(t-h_i)B_i^\T,t\ge 0\\
  X_0(s)&=&\Xi(s)\in\psd^n,\ s\in[-\bar h,0]
\end{array}
\end{equation}
where the delays $h_i$'s are positive scalars with $\textstyle\bar h:=\max_i\{h_i\}$, and $\Xi\in C([-\bar h,0],\psd^n)$ is the initial condition. As for the system \eqref{eq:matdiffeqCT_LTV}, it can be proven that this system leaves the cone of positive semidefinite invariant using exactly the same arguments. Define, moreover, the operators
\begin{equation}\label{eq:operatorLt:h}
  \begin{array}{rccl}
    \mathcal{H}_i:&\sm^n&\mapsto&\sm^n\\
    & X&\mapsto & B_i  XB_i^\T.
  \end{array}
\end{equation}
and their adjoints
\begin{equation}\label{eq:operatorLtadjoint:h}
  \begin{array}{rccl}
    \mathcal{H}^*_i:&\sm^n&\mapsto&\sm^n\\
    & X&\mapsto & B_i^\T XB_i.
  \end{array}
\end{equation}

The following result demonstrates that the same delay-independent stability property holds for the system \eqref{eq:matdiffeqCT_LTV:delay} as for linear time-invariant positive systems \citep{Haddad:04,Kaczorek:09,Briat:11h,Briat:16b}:
\begin{theorem}\label{th:CT:delay}
The following statements are equivalent:
  \begin{enumerate}[(a)]
    \item\label{st:CT:delay1} The system \eqref{eq:matdiffeqCT_LTV:delay} is exponentially stable for all delay values $h_i\ge0$, $i=1,\ldots,N$.
    \item\label{st:CT:delay1b} The system \eqref{eq:matdiffeqCT_LTV:delay} with zero delays (i.e. $h_i=0$, $i=1,\ldots,N$) is exponentially stable.
    \item\label{st:CT:delay2} There exist matrices $P,Q_i\in\pd^n$, $i=1,\ldots,N$, such that the conditions
    \begin{equation}\label{eq:kdospdosadadk}
   \mathcal{C}^*(P)+\sum_{i=1}^NQ_i\prec0\ \textnormal{and }\mathcal{H}_i^*(P)-Q_i\preceq0,\ i=1,\ldots,N
\end{equation}
    hold.
    \item\label{st:CT:delay3} We have that $\sp\left(\mathcal{C}+\sum_{i=1}^N\mathcal{H}_i\right)\subset\mathbb{C}_{<0}$ where
  \begin{equation}
        \sp\left(\mathcal{C}+\sum_{i=1}^N\mathcal{H}_i\right):=\left\{\lambda\in\mathbb{C}: \mathcal{C}(X)+\sum_{i=1}^N\mathcal{H}_i(X)=\lambda X\textnormal{ for some }X\in\mathbb{S}^{n},X\ne0\right\}.
  \end{equation}
  \end{enumerate}
\end{theorem}
\begin{proof}
We know from Theorem \ref{th:stabMatCT_LTI} that the statements \eqref{st:CT:delay1b} and \eqref{st:CT:delay3} are equivalent. It is also immediate to see that the statement \eqref{st:CT:delay1} implies the statement \eqref{st:CT:delay1b}. So, we need to prove the two remaining implications.\\

  \noindent\textbf{Proof that the statement \eqref{st:CT:delay2} implies the statement \eqref{st:CT:delay1}.} Let us consider the following functional
\begin{equation}\label{eq:LKF:CT}
  V(X_t)=\langle P,X_t(0)\rangle+\sum_{i=1}^N\int_{-h_i}^0\langle Q_i,X_t(s)\rangle\ds
\end{equation}
where $X_t(s)=X(t+s),s\in[-\bar h,0]$.
for which we have
\begin{equation}
  \lmin(P)||X_t(0)||_*\le V(X_t)\le (\lmax(P)+\lmax(Q))||X_t||_{c,*}
\end{equation}
where $||X_t||_{c,*}:=\max_{s\in[-\bar h,0]}||X_t(s)||_*$. Computing the derivative yields
\begin{equation}
  \begin{array}{rcl}
    \dot{V}(X_t)   &=&    \langle P,\dot{X}(t)\rangle+\sum_{i=1}^N\left[\langle Q_i,X(t)\rangle-\langle Q_i,X(t-h_i)\rangle\right]\\
                                &=&    \langle P,\mathcal{C}(X(t))\rangle+\sum_{i=1}^N\left[\langle P,\mathcal{H}_i(X(t-h_i))\rangle+\langle Q_i,X(t)\rangle-\langle Q_i,X(t-h_i)\rangle\right]\\
                                &=&    \left\langle \mathcal{C}^*(P)+\sum_{i=1}^NQ_i, X(t)\right\rangle+\sum_{i=1}^N\langle \mathcal{H}_i^*(P)-Q_i,X(t-h_i)\rangle
  \end{array}
\end{equation}
Under the conditions of the statement \eqref{st:CT:delay2}, there exists a small enough $\eps>0$ such that $\textstyle\mathcal{C}^*(P)+\sum_{i=1}^NQ_i\preceq -\eps I$ which implies
\begin{equation}
  \begin{array}{rcl}
    \dot{V}(X_t)   &\le&   -\eps ||X_t(0)||_*
  \end{array}
\end{equation}
Therefore, the functional \eqref{eq:LKF:CT} is a Lyapunov-Krasovskii functional \citep{GuKC:03} for the system \eqref{eq:matdiffeqCT_LTV:delay} and the system is exponentially stable regardless the value of the delays as the conditions do not depend on the values of the delays. This proves the implication.\\

\noindent\textbf{Proof that the statement \eqref{st:CT:delay1b} implies the statement \eqref{st:CT:delay2}.} Assume that the system \eqref{eq:matdiffeqCT_LTV:delay} with zero delays (i.e. $h_i=0$, $i=1,\ldots,N$) is exponentially stable. This is equivalent to saying that the system
\begin{equation}
  \dot{X}(t)=\mathcal{C}(X(t))+\sum_{i=1}^N\mathcal{H}_i(X(t))
\end{equation}
is exponentially stable. This implies that there exist a matrix $P\in\pd^n$ and an $\eps>0$ such that
\begin{equation}\label{eq:kd;lsakdl;sa;ldksal;dk;askd}
  \mathcal{C}^*(P)+\sum_{i=1}^N\mathcal{H}_i^*(P)\preceq-2\eps I.
\end{equation}
Define now
\begin{equation}
  Q_i=\mathcal{H}_i^*(P)+\dfrac{\eps}{N}I_n\succ0
\end{equation}
and substituting this value in \eqref{eq:kdospdosadadk} yields
\begin{equation}
  \begin{array}{rcl}
     \mathcal{C}^*(P)+\sum_{i=1}^NQ_i&=&\mathcal{C}^*(P)+\sum_{i=1}^N\mathcal{H}_i^*(P)\preceq-\eps I\prec0,\\
     \mathcal{H}_i^*(P)-Q_i&=&-\dfrac{\eps}{N}I_n\preceq0,
  \end{array}
\end{equation}
which implies that the conditions in the statement \eqref{st:CT:delay2} hold, thereby proving the desired result. The proof is completed.
\end{proof}

The above result is the semidefinite analogue of the results obtained for linear positive systems with discrete delays -- see e.g. \citep{Haddad:04}. The stability properties do not depend on the value of the delays and the system is exponentially stable for all possible values of the delays if and only if the system with zero-delays is exponentially stable.  It also admits an extension to the periodic systems case using a periodic Lyapunov-Krasovskii functional under the same assumptions as in \citep{Briat:19:IJC}. This extension is omitted for brevity.}

\subsection{Applications}

\blue{The objective of this section is to demonstrate that many results from the literature can be recovered and extended using the obtained results. To avoid repetitions, in all the sections below, $x,x_0\in\mathbb{R}^n$ denote the state of the system and the initial condition, $t_0\ge0$ is the initial time, $W(t)=(W_1(t),\ldots,W_{N}(t))$ is a vector of $N$ independent Wiener processes which are also independent of the state $x(t)$ and all the other signals possibly involved in the system. To any of the systems, we can associate a filtered probability space $(\Omega,\mathcal{F},\mathcal{F}_{t},\P)$ with natural filtration satisfying the usual conditions.

The first example pertains to the analysis of stochastic LTV systems with Poissonian jumps, the second one to the analysis of diffusion processes with delays, the third one is to the analysis of time-varying Markov jump linear systems, and the last one to the analysis and the control of time-varying sampled-data systems with Poissonian sampling.}

\subsubsection{Stochastic LTV systems with Poissonian jumps}

\black{Let us define the class of stochastic LTV with Poissonian jumps as
\begin{equation}\label{eq:syst:CT_general_LTV}
  \begin{array}{rcl}
    \d x(t)&=&A_0(t)x(t)\dt+\sum_{i=1}^{N}A_i(t)x(t)\dW_i(t)+\sum_{i=0}^N(M_i(t)-I)x(t)\d\eta_i(t),t\ge t_0\\
    x(t_0)&=&x_0
  \end{array}
\end{equation}
where $(\eta_0(t),\eta_1(t),\ldots,\eta_N(t))$ is a vector of $N+1$ independent of Poisson processes with rates $(\lambda_0,\lambda_1,\ldots,\lambda_N)$ and which are independent of $x(t)$, $W(t)$ and of the matrix-valued functions $A_i,M_i:\mathbb{R}_{\ge0}\to\mathbb{R}^{n\times n}$, $i=0,\ldots,N$, which are assumed to be piecewise continuous and bounded.}

Several stability notions exist for the system above \citep{Khasminskii:12} but we will be mostly interested in the uniform exponential mean-square stability defined below:
\black{\begin{definition}\label{def;MSS1}
  The system \eqref{eq:syst:CT_general_LTV} is uniformly exponentially mean-square stable if there exist some scalars $\alpha>0$ and $\beta\ge1$ such that
  \begin{equation}
    \E[||x(t)||_2^2]\le\beta e^{-\alpha(t-t_0)}\E[||x_0||_2^2]
  \end{equation}
  holds for all $t\ge t_0$, all $t_0\ge0$, and all $x_0\in\mathbb{R}^n$.
\end{definition}

One can observe that $\E[||x(t)||_2^2]=\trace\E[x(t)x(t)^{\T}]$ which means that this notion of stability is about the exponential convergence of the sum of the (nonnegative) eigenvalues of the \blue{seconder-order moment} matrix associated with the system \eqref{eq:syst:CT_general_LTV} to zero. This remark leads us to the following result:
\begin{theorem}
  The following statements are equivalent:
  \begin{enumerate}[(a)]
    \item The system \eqref{eq:syst:CT_general_LTV} is uniformly exponentially mean-square stable.
    \item There exist a differentiable matrix-valued function $P:\mathbb{R}_{\ge0}\mapsto\mathbb{S}^n_{\succ0}$ and a continuous matrix-valued function $Q:\mathbb{R}_{\ge0}\mapsto\mathbb{S}^n_{\succ0}$ such that
  \begin{equation}
    \begin{array}{rcccl}
      \alpha_1I&\preceq&P(t)&\preceq&\alpha_2I,\\
      \beta_1I&\preceq&Q(t)&\preceq&\beta_2I,
    \end{array}
  \end{equation}
    and such that Lyapunov differential equation
    \begin{equation}\label{eq:dkpspokdskpodkpokopsdzkpoko;a;dks;lkk}
    \begin{array}{rcl}
    \dot{P}(t)+\He[P(t)A_0(t)]+\sum_{i=1}^NA_i(t)^\T P(t)A_i(t)+\sum_{i=0}^N\lambda_i\Delta_i^*(t)=-Q(t)
    \end{array}
  \end{equation}
  where $\Delta_i^*(t):=M_i(t)^{\T}P(t)M_i(t)-P(t)$, $i=0,\ldots,N$, hold for some positive constants $\alpha_1,\alpha_2,\beta_1,\beta_2$ and for all $t\ge t_0$ and all $t_0\ge0$.
  \end{enumerate}
\end{theorem}
\begin{proof}
  The system \eqref{eq:syst:CT_general_LTV} is uniformly exponentially mean-square stable if and only if the dynamics of the \blue{seconder-order moment} matrix $X(t):=\E[x(t)x(t)^{\T}]$ given by
    \begin{equation}
      \begin{array}{rcl}
        \dot{X}(t)&=&A_0(t)X(t)+X(t)A_0(t)^{\T}+\sum_{i=1}^NA_i(t)X(t)A_i(t)^{\T}+\sum_{i=0}^N\lambda_i\Delta_i(t),\ t\ge t_0\\
        X(t_0)&=&X_0:=\E[x_0x_0^{\T}]
      \end{array}
    \end{equation}
    where $\Delta_i(t):=M_i(t)X(t)M_i(t)^{\T}-X(t)$, $i=0,\ldots,N$, is uniformly exponentially stable. Applying then Theorem \ref{th:stabMatCT_LTV} yields the result.
\end{proof}}

\blue{We can extend the system \eqref{eq:syst:CT_general_LTV} to include control inputs as
\begin{equation}\label{eq:syst:CT_general_LTV_u}
  \begin{array}{rcl}
    \d x(t)&=&(A_0(t)x(t)+B_0(t)u_0(t))\dt+\sum_{i=1}^{N}(A_i(t)x(t)+B_i(t)u_i(t))\dW_i(t)\\
    &&\qquad+\sum_{i=0}^N[(M_i(t)-I)x(t)+N_i(t)u_i(t)]\d\eta_i(t),t\ge t_0\\
    x(t_0)&=&x_0
  \end{array}
\end{equation}
where $u_i(\cdot)\in\mathbb{R}^{m_i}$, $i=0,\ldots,N$, and the matrix-valued functions $B_i,N_i:\mathbb{R}_{\ge0}\mapsto\mathbb{R}^{n\times m_i}$, $i=0,\ldots,N$, are assumed to be piecewise constant and bounded. Considering the control laws
\begin{equation}
  u_i(t)=K_i(t)x(t),\ i=0,\ldots,N
\end{equation}
and making the substitution $A_i\leftarrow A_i+B_iK_i$ and $M_i\leftarrow M_i+N_iK_i$ in \eqref{eq:dkpspokdskpodkpokopsdzkpoko;a;dks;lkk} yields an operator of the form \eqref{eq:operatorLtB1}. This leads us to the following result which can be seen as a generalization of the those in \citep{Briat:13d,Briat:15i}:
\begin{theorem}
  The following statements are equivalent:
  \begin{enumerate}[(a)]
    \item The system \eqref{eq:syst:CT_general_LTV_u} is uniformly exponentially mean-square stabilizable.
    \item There exist a differentiable matrix-valued function $Q:\mathbb{R}_{\ge0}\mapsto\mathbb{S}^n_{\succ0}$, $\alpha_1I\preceq P(t)\preceq\alpha_2I$, and matrix-valued functions $U_i:\mathbb{R}_{\ge0}\mapsto\mathbb{R}^{m_i\times n}$, $i=0,\ldots,N$, such that the Lyapunov differential inequality
    \begin{equation}
    \begin{bmatrix}
      -\dot{Q}(t)+\He[A_0(t)Q(t)+B_0(t)U_0(t)]-\sum_{i=0}^N\lambda_i Q(t) & \row_{i=1}^N[\Psi_i(t)^{\T}] & \row_{i=0}^N[\lambda_i\Phi_i(t)^{\T}]\\
      \star & -I_N\otimes Q(t) & 0\\
      \star & \star & -\diag_{i=0}^N(\lambda_i)\otimes Q(t)
    \end{bmatrix}\preceq -\alpha_3 I
  \end{equation}
  holds for some positive constants $\alpha_1,\alpha_2,\alpha_3$ and for all $t\ge t_0$ and all $t_0\ge0$ where $\Psi_i(t):=A_i(t)Q(t)+B_i(t)U_i(t)$ and $\Phi_i(t):=M_i(t)Q(t)+N_i(t)U_i(t)$.
  \end{enumerate}
\end{theorem}}
\begin{proof}
  This is a direct application of Theorem \ref{th:CT:stabz}.
\end{proof}

\subsubsection{A class of diffusion processes with delays}

\black{Let us consider the following class of diffusion systems with delays
\begin{equation}\label{eq:delaystoch}
\begin{array}{rcl}
  \dx(t)&=&A_0x(t)\dt+\sum_{i=1}^NA_ix(t-h_i)\dW_i(t)\\
  x(s)&=&\phi(s),\ s\in[-\bar h,0]
\end{array}
\end{equation}
where $\phi\in C([-\bar h,0],\mathbb{R}^n)$ is the initial condition. The notion of stability we will be interested here is also the mean square exponential stability which we adapt to systems with delays as
\begin{definition}
  The system \eqref{eq:delaystoch} is mean-square  exponentially stable if there exist some scalars $\alpha>0$ and $\beta>1$ such that
  \begin{equation}
    \E[||x(t)||_2^2]\le\beta e^{-\alpha t}\E\left[\max_{s\in[-\bar h,0]}||\phi(s)||_2^2\right]
  \end{equation}
  holds for all $t\ge0$ and all initial condition $\phi\in C([-\bar h,0],\mathbb{R}^n)$.
\end{definition}

Once again, this is related to the exponential convergence of the eigenvalues of the \blue{seconder-order moment} matrix, which then leads us to the following result that can be seen as a generalization of Theorem 5 in \citep{Tanaka:13b}:
\begin{theorem}
The following statements are equivalent:
\begin{enumerate}[(a)]
    \item The system \eqref{eq:delaystoch} is mean-square  exponentially stable for all $h_i\ge0$, $i=1,\ldots,N$.
    \item The system  \eqref{eq:delaystoch}  with zero-delays (i.e. with $h_i=0$, $i=1,\ldots,N$) is mean-square exponentially stable.
\end{enumerate}
\end{theorem}
\begin{proof}
    Let $X(t):=\E[x(t)x(t)^{\T}]$ to be the \blue{seconder-order moment} matrix of the system \eqref{eq:delaystoch}. The dynamics of this matrix is described by
\begin{equation}
  \dot{X}(t)=A_0X(t)+X(t)A_0^{\T}+\sum_{i=1}^NA_i^{\T}X(t-h_i)A_i.
\end{equation}
Applying now Theorem \ref{th:CT:delay} yields the result.
\end{proof}}

\subsubsection{Time-varying Markov jump linear systems}

Let us consider the following class of systems described by a linear stochastic differential equation subject to Markovian jumps
\begin{equation}\label{eq:CT:Markov}
\begin{array}{rcl}
  dx(t)&=&A_{\sigma(t)}(t)x(t)\dt+\sum_{i=1}^NB_{\sigma(t),i}(t)x(t)\dW_i(t),\ t\ge t_0\vspace{-3mm}\\
  x(t_0)&=&x_0,\\
  \sigma(t_0)&=&\sigma_0
\end{array}
\end{equation}
where  $\sigma_0$ is the initial value of the parameter $\sigma:[t_0,\infty)\mapsto\{1,\ldots,M\}$ which changes value following a finite Markov chain as
\begin{equation}\label{eq:CT:Markov2}
  \P(\sigma(t+h)=j|\sigma(t)=i)=\left\{\begin{array}{rcl}
    \pi_{ij}h+o(h),&& i\ne j,\\
    1+\pi_{ii}h+o(h),&& i=j,
  \end{array}\right.
\end{equation}
where $\pi_{ij}\ge0$ for all $i\ne j$ and $\pi_{ii}=-\textstyle \sum_{j|j\ne i}\pi_{ij}$. The matrix-valued functions $A_i:\mathbb{R}_{\ge0}\mapsto\mathbb{R}^{n\times n}$, $i=1,\ldots,M$, are assumed to be piecewise continuous and bounded.  \bblue{We consider here the concept of stability called \emph{uniformly exponential mean-square stability in conditioning} as defined in \citep{Dragan:02,Dragan:13}:
\begin{definition}\label{def;MSS1b}
  The system \eqref{eq:CT:Markov}\eqref{eq:CT:Markov2} is uniformly exponentially mean-square stable in conditioning if there exist some scalars $\alpha>0$ and $\beta\ge1$ such that
  \begin{equation}
    \E[||x(t)||_2^2|\sigma(t_0)]\le\beta e^{-\alpha(t-t_0)}\E[||x_0||_2^2]
  \end{equation}
  holds for all $t\ge t_0$, all $t_0\ge0$, all $x_0\in\mathbb{R}^n$ and all initial probability distribution $\pi(t_0)$ of the Markov process defined by \eqref{eq:CT:Markov2}.
\end{definition}

This leads to the following result\footnote{\bblue{It was brought to the attention of the author after publication that this result was previously published in \citep{Dragan:02,Dragan:13}.}} which can be seen as a generalization of the results in \citep{Ji:90,Costa:13}:}

\black{\begin{theorem}\label{th:Jump:CT}
  The system \eqref{eq:CT:Markov}-\eqref{eq:CT:Markov2} is \bblue{uniformly exponentially mean-square stable in conditioning} if and only if there exist differentiable matrix-valued functions $P_i:\mathbb{R}_{\ge0}\mapsto\pd^n$, $i=1,\ldots,M$, and continuous matrix-valued functions $Q_i:\mathbb{R}_{\ge0}\mapsto\pd^n$, $i=1,\ldots,M$, such that
  \begin{equation}
    \begin{array}{rcccl}
      \alpha_1I&\preceq&P_i(t)&\preceq&\alpha_2I,\\
      \beta_1I&\preceq&Q_i(t)&\preceq&\beta_2I,
    \end{array}
  \end{equation}
  and the coupled Lyapunov differential equations
\begin{equation}\label{eq:dl;skdl;kasdk;ldk;sdk;ak}
  \dot{P}_i(t)+A_i(t)^{\T}P_i(t)+P_i(t)A_i(t)+\sum_{j=1}^NB_{i,j}^{\T}P_i(t)B_{i,j}+\sum_{\substack{j=1\\j\ne i}}^M\pi_{ij}P_j(t)=-Q_i(t)
\end{equation}
hold for all $ i=1,\ldots,M$ and some $\alpha_1.\alpha_2,\beta_1,\beta_2>0$.
\end{theorem}
\begin{proof}
  Let $\mathds{1}_i(\sigma(t))$ be the indicator function of the state $i$ defined as $\mathds{1}_i(\sigma(t))=1$  if $\sigma(t)=i$, and 0 otherwise. Then, defining $X_i(t)=\E[x(t)x(t)^{\T}\mathds{1}_i(\sigma(t))]$, we get that
\begin{equation}\label{eq:dksalkd;lsakd;lksa;dksa;ldka;skd;}
  \dot{X}_i(t)=A_i(t)X_i(t)+X_i(t)A_i(t)^{\T}+\sum_{j=1}^NB_{i,j}(t)X_i(t)B_{i,j}(t)^{\T}+\sum_{j=1}^M\pi_{ji}X_j(t)
\end{equation}
for all $i=1,\ldots,M$. While it is readily seen that while $\E[x(t)x(t)^{\T}]=\textstyle\sum_{i=1}^MX_i(t)$, there is no closed-form expression for the derivative of the \blue{seconder-order moment} matrix associated with the system. Let $E_i$ be the matrix composed of the columns $\{e_1,\ldots,e_M\}$ in cyclic ascending order where $e_1$ is the $i$-th column and define $\bar \Pi$ as $[\bar \Pi]_{ij}=\sqrt{\pi_{ji}}$. Then, we have that
\begin{equation}
  \diag_{i=1}^M\left(\sum_{\substack{j=1\\ j\ne i}}^M\pi_{ji}X_j(t)\right)=\sum_{j=2}^M\bar{A}_jX(t)\bar{A}_j^{\T}
\end{equation}
where $\bar{A}_j= (E_i\odot\bar{\Pi})\otimes I_n$, $j=2,\ldots,M$. Using this expression, one can rewrite all the matrix-valued differential equations \eqref{eq:dksalkd;lsakd;lksa;dksa;ldka;skd;} in compact form as
\begin{equation}
  \dot{X}(t)=\bar{A}_0(t)X(t)+X(t)\bar{A}_0(t)^{\T}+\sum_{j=1}^N\bar B_{j}(t)X(t)\bar B_{j}(t)^{\T}+\sum_{j=2}^M\bar{A}_jX(t)\bar{A}_j^{\T}
\end{equation}
where $X(t):=\textstyle\diag_{i=1}^M\{X_i(t)\}$,
\begin{equation}
  \begin{array}{rclcrcl}
    \bar{A}_0(t)&=&\diag_{i=1}^M\left\{A_i(t)+\frac{\pi_{ii}}{2}I\right\},&\textnormal{and}&  \bar B_{j}(t)&=&\diag_{i=1}^M(B_{i,j}(t)),\\
  \end{array}
\end{equation}

Using Theorem \ref{th:stabMatCT_LTV}, we obtain the stability condition
\begin{equation}\label{eq:kdskl;dksa;ldkl;askd;lsk;dkasd;lsak}
  \dot{P}(t)+\bar{A}_0(t)^{\T}P(t)+P(t)\bar{A}_0(t)+\sum_{j=1}^N\bar B_{j}(t)^{\T}P(t)\bar B_{j}(t)+\sum_{j=1}^M\bar{A}_j^{\T}P(t)\bar{A}_j=-Q(t).
\end{equation}
As the state of this system has a block-diagonal structure, the state space is actually $\mathbb{S}^n_{\succeq0}\times\ldots\mathbb{S}^n_{\succeq0}$ ($M$ times) and not the full space $\mathbb{S}^{nM}_{\succeq0}$. This allows us to also restrict ourselves to block-diagonal matrices $P(t)=\textstyle\diag_{i=1}^M(P_i(t))$ and $Q(t)=\textstyle\diag_{i=1}^M(Q_i(t))$ which all satisfy the conditions of Theorem \ref{th:stabMatCT_LTV}. Exploiting the block-diagonal structure yields
\begin{equation}
  \sum_{j=2}^M\bar{A}_j^{\T}P(t)\bar{A}_j=\diag_{i=1}^M\left(\sum_{\substack{j=1\\j\ne i}}\pi_{ij}P_j(t)\right).
\end{equation}
together with  the conditions \eqref{eq:dl;skdl;kasdk;ldk;sdk;ak} after having expanded the other terms in \eqref{eq:kdskl;dksa;ldkl;askd;lsk;dkasd;lsak}.
\end{proof}}

\subsubsection{Time-varying sampled-data systems with Poissonian sampling}

\blue{
Let us consider here the following LTV stochastic process
\begin{equation}\label{eq:syst:Impulsive_SD1}
  \begin{array}{rcl}
    \d x(t)&=&(A_0(t)x(t)+B_0(t)u(t))\dt+\sum_{i=1}^{N}(A_i(t)x(t)+B_i(t)u(t))\d W_i(t)\\
    x(t^0)&=&x_0
  \end{array}
\end{equation}
where $A_i:\mathbb{R}_{\ge0}\mapsto\mathbb{R}^{n\times n}$ and $B_i:\mathbb{R}_{\ge0}\mapsto\mathbb{R}^{n\times m}$, $i=0,\ldots,N$, are assumed to be piecewise continuous and bounded. We also assume that this system is controlled using the following time-varying sample-data state-feedback control law
\begin{equation}\label{eq:syst:Impulsive_SD2}
  u(t)=K_1(t_k)x(t_k)+K_2(t_k)u(t_k),\ t\in(t_k,t_{k+1}].
\end{equation}
As for the other systems, $x,x_0\in\mathbb{R}^n$ are the state of the system and the initial condition, and $t^0\ge0$ is the initial time. The Wiener processes $W_1(t),\ldots,W_{N}(t)$ are assumed to be independent of each other and of the state $x(t)$. The matrix-valued functions describing the system are also assumed to be uniformly bounded. Let us now assume that the sampling instants $t_{k+1}-t_k$ is exponentially distributed with rate $\lambda>0$; i.e. $T_k\sim \textnormal{Exp}(\lambda)$.

\begin{theorem}
  The following statements are equivalent:
  \begin{enumerate}[(a)]
    \item The stochastic LTV sampled-data system \eqref{eq:syst:Impulsive_SD1}-\eqref{eq:syst:Impulsive_SD2} with Poissonian sampling with rate $\lambda>0$ is uniformly globally exponentially stable.
    \item There exists a differentiable matrix-valued function $P:\mathbb{R}_{\ge0}\mapsto\mathbb{S}^n_{\succ0}$, $\alpha_1I\preceq P(t)\preceq\alpha_2I$, such that the differential matrix inequality
    \begin{equation}
    \begin{array}{rcl}
    \dot{P}(t)+\He[P(t)\bar A_0(t)]+\sum_{i=1}^N\bar A_i(t)^\T P(t)\bar A_i(t)+\lambda(\bar J(t)^{\T}P(t)\bar J(t)-P(t))\preceq-\alpha_3I
    \end{array}
  \end{equation}
  holds for some positive constants $\alpha_1,\alpha_2,\alpha_3$ and for all $t\ge t_0$ and all $t_0\ge0$ where
  \begin{equation}
    \bar A_0(t):=\begin{bmatrix}
      A_0(t) & B_0(t)\\
      0 & 0
    \end{bmatrix},\ \bar A_i(t):=\begin{bmatrix}
      A_i(t) & B_i(t)\\
      0 & 0
    \end{bmatrix},\ \bar J(t):=\begin{bmatrix} I & 0\\
      K_1(t) & K_2(t)
    \end{bmatrix}.
  \end{equation}
  \item There exist a differentiable matrix-valued function $Q:\mathbb{R}_{\ge0}\mapsto\mathbb{S}^n_{\succ0}$,  $\alpha_1I\preceq Q(t)\preceq\alpha_2I$, and a matrix-valued function $U:\mathbb{R}_{\ge0}\mapsto\mathbb{R}^{m\times (n+m)}$, $i=1,\ldots,N$, such that the differential matrix inequality
    \begin{equation}
    \begin{bmatrix}
      -\dot{Q}(t)+\He[\bar A_0(t)Q]-\lambda Q(t) & \lambda\left(\bar{J}_0(t)Q(t)+\bar{J}_b(t)U(t)\right)^{\T} & \row_{i=1}^N[Q(t)A_i(t)^{\T}]\\
      \star & -\lambda Q(t) & 0\\
      \star & \star & -I_n\otimes Q(t)
    \end{bmatrix}\preceq-\alpha_3I
  \end{equation}
  holds for some positive constants $\alpha_1,\alpha_2,\alpha_3$ and for all $t\ge t_0$ and all $t_0\ge0$ where
    \begin{equation}
    \bar J_0(t):=\begin{bmatrix}
      I_n & 0\\
      0 & 0
    \end{bmatrix},\ \bar J_b(t):=\begin{bmatrix}
      0\\
      I_m
    \end{bmatrix}.
  \end{equation}
  Moreover, suitable gains are given by $\begin{bmatrix}
    K_1(t) & K_2(t)
  \end{bmatrix}=U(t)Q(t)^{-1}$.
  \end{enumerate}
\end{theorem}

\begin{proof}
  The  system \eqref{eq:syst:Impulsive_SD1}-\eqref{eq:syst:Impulsive_SD2} can be reformulated as the following impulsive system
\begin{equation}
  \d z(t)=\bar{A}_0(t)z(t)\dt+\sum_{i=1}^N\bar{A}_i(t)z(t)\d W_i(t)
\end{equation}
and
\begin{equation}
  z(t_k^+)=\bar J(t_k)z(t_k)
\end{equation}
where $\bar A_i(t)$, $i=1,\ldots,N$, and $\bar J(t_k)$ are defined in the result. The seconder-order moment system associated with the above impulsive system is given by
\begin{equation}
  \dot{X}(t)=\bar{A}_0(t)X(t)+X(t)\bar{A}_0(t)^{\T}+\sum_{i=1}^N\bar{A}_i(t)X(t)\bar{A}_i(t)^{\T}+\lambda(\bar J(t)X(t)\bar J(t)^{\T}-X(t)).
\end{equation}
Applying then Theorem \ref{th:stabMatCT_LTV} yields the result of the second statement whereas the use of Theorem \ref{th:CT:stabz} leads to the third statement.
\end{proof}
}

\section{Linear Matrix-Valued Symmetric Discrete-Time Systems}\label{sec:DT}

\blue{This section is devoted to the definition of linear matrix-valued symmetric discrete-time systems that evolves on the cone of positive semidefinite matrices together with their associated generators and state-transition operator. Using those definitions a necessary and sufficient condition for their uniform exponential stability is obtained in Theorem \ref{th:stabMatDT_LTV} in terms of a matrix-valued difference equation or a matrix-valued difference inequality. This result is then specialized to the LTI case in Theorem \ref{th:stabMatDT_LTI} where additional spectral conditions are also obtained. Convex stabilizability and stabilization conditions are also derived in Theorem \ref{th:DT:stabz} and take the form of difference linear matrix inequalities. The case of systems with delays is also considered  and a necessary and sufficient stability condition, generalizing existing ones, is also obtained in Theorem \ref{th:DT:delay}. As in the continuous-time case, the simplicity and the versatility of the approach as a general tool for the analysis of (stochastic) discrete-time systems is illustrated through several examples, namely on linear time-varying stochastic systems subject to difference types of stochastic finite-variance parameters, linear time-varying stochastic systems with delays, and time-varying systems with Markov jump parameters.}

\subsection{Preliminaries}

Let us consider here the following class of matrix-valued discrete-time dynamical systems
\begin{equation}\label{eq:matdiffeqDT_LTV}
\begin{array}{rcl}
  X(k+1)&=&\sum_{i=0}^NJ_i(k) X(k)J_i(k)^\T,k\ge k_0\\
  X(k_0)&=&X_0\in\psd^n
\end{array}
\end{equation}
where the matrix-valued functions $J_i:\mathbb{Z}_{\ge0}\mapsto\mathbb{R}^{n\times n}$, $i=0,\ldots,N$, are bounded. It is clear from the above expression that there is a unique solution to this difference equation, which we denote by $X(k,k_0,X_0)$. This solution can be written as
\begin{equation}\label{eq:PHIDTMATLTV}
  X(k)=\mathcal{D}_{k,\ell}(X_\ell),\ k\ge\ell\ge k_0
\end{equation}
where
\begin{equation}\label{eq:PHIDTMATLTV2}
\mathcal{D}_{k,\ell}(X)=\left\{\begin{array}{lcl}
  X&&\textnormal{if }k=\ell\\
  \mathcal{D}_\ell(X)&&\textnormal{if }k=\ell+1\\
  \left(\mathcal{D}_{k-1}\circ\ldots\circ\mathcal{D}_\ell\right)(X)&&\textnormal{if }k>\ell+1
\end{array}\right.
\end{equation}
and
\begin{equation}\label{eq:operatorLdt}
  \begin{array}{rccl}
    \mathcal{D}_k:&\sm^n&\mapsto&\sm^n\\
    & X&\mapsto & \sum_{i=0}^NJ_i(k) XJ_i(k)^\T.
  \end{array}
\end{equation}
The adjoint of the operator $\mathcal{D}_k$, denoted by $\mathcal{D}_k^*$,  is given by
\begin{equation}\label{eq:operatorLdtadjoint}
  \begin{array}{rccl}
    \mathcal{D}^*_k:&\sm^n&\mapsto&\sm^n\\
    & X&\mapsto & \sum_{i=0}^NJ_i(k)^\T XJ_i(k).
  \end{array}
\end{equation}

\black{The following result, which is the discrete-time analogue of Proposition \ref{prop:positivityCT}, shows that the system  \eqref{eq:matdiffeqDT_LTV}  leaves the cone $\mathbb{S}^n_{\succeq0}$ invariant:
\begin{proposition}\label{prop:positivityDT}
  The matrix-valued difference equation \eqref{eq:matdiffeqDT_LTV} leaves the cone of positive semidefinite matrices invariant, that is, for any $X_0\succeq0$, we have that $X(k,k_0,X_0)\succeq0$ for all $k\ge k_0$ and all $k_0\ge0$.
\end{proposition}

\begin{proof}
  The proof follows from the same lines as the proof of Proposition \ref{prop:positivityCT} and is thus omitted.
\end{proof}

\begin{definition} \label{def:DTsystemMat}
  The zero solution of the  system \eqref{eq:matdiffeqDT_LTV} is \textbf{globally uniformly exponentially stable with rate $\rho$} if there exist some constants $\beta\ge1$ and $\rho\in(0,1)$ such that
  \begin{equation}
                  ||X(k,k_0,X_0)||\le\beta\rho^{k-k_0}||X_0||
  \end{equation}
  for all $k\ge k_0$, all $k_0\in\Znn$, and all $X_0\in\psd^{n}$.
\end{definition}}

\subsection{Stability Analysis}

\black{With the previous definitions and results in mind, we are now able to state the following result\footnote{It was brought to the attention of the author after publication of the paper that this result can be seen as a particular case of Theorem 2.4 in \cite{Dragan:10}.} characterizing the uniform exponential stability of the system \eqref{eq:matdiffeqDT_LTV}:

\begin{theorem}[LTV Case]\label{th:stabMatDT_LTV}
  The following statements are equivalent:
  \begin{enumerate}[(a)]
  \item\label{st:stabMatDT_LTV1} The system \eqref{eq:matdiffeqDT_LTV} is uniformly exponentially stable.
    %
    %
    \item\label{st:stabMatDT_LTV3}  There exist some matrix-valued functions $P,Q:\mathbb{Z}_{\ge0}\mapsto\mathbb{S}^n_{\succ0}$ such that
    \begin{equation}
    \begin{array}{rcccl}
      \alpha_1I&\preceq& P(k)&\preceq &\alpha_2I,\\
      \beta_1I&\preceq &Q(k)&\preceq &\beta_2I,
    \end{array}
    \end{equation}
    and such that \textbf{Lyapunov difference equation}
    \begin{equation}
    \mathcal{D}_k^*(P(k+1))-P(k)+Q(k)=0
  \end{equation}
  hold for some positive constants $\alpha_1,\alpha_2,\beta_1,\beta_2$ and for all $k\ge0$.
    \item\label{st:stabMatDT_LTV4}  There exists a matrix-valued function $P:\mathbb{Z}_{\ge0}\mapsto\mathbb{S}^n_{\succ0}$ such that the
    \begin{equation}
      \alpha_1I\preceq P(k)\preceq \alpha_2I
    \end{equation}
    and such that \textbf{Lyapunov difference inequality}
    \begin{equation}
    \mathcal{D}_k^*(P(k+1))-P(k)\preceq-\alpha_3 I
  \end{equation}
  hold for some positive constants $\alpha_1,\alpha_2,\alpha_3$ and for all $k\ge0$.
  \end{enumerate}
\end{theorem}}

\begin{proof}
\black{It is immediate to see that the statements \eqref{st:stabMatDT_LTV3} and \eqref{st:stabMatDT_LTV4} are equivalent.\\

\noindent\textbf{Proof that the statement (b) implies the statement (a).} Define the function $V(k,X)=\langle P(k),X\rangle$ where $P(\cdot)$ satisfies the conditions of statement (b). Since $P(k)$ is positive definite, uniformly bounded and uniformly bounded away from 0, then, from Proposition \ref{prop:PQ}, we have that $V$ is $\psd^n$-copositive definite, radially unbounded and decrescent. Evaluating the successive difference of the function $V$ along the trajectories of the system \eqref{eq:matdiffeqDT_LTV} yields
\begin{equation}
  \begin{array}{rcl}
    V(k+1,X(k+1)-V(k,X(k)))&=&\langle P(k+1),X(k+1)\rangle-\langle P(k),X(k)\rangle\\
                                &=&\langle P(k+1),\mathcal{D}_k(X(k))\rangle-\langle P(k),X(k)\rangle\\
                                &=&\langle \mathcal{D}^*_k(P(k+1))-P(k),X(k)\rangle\\
                                &=&-\langle Q(k),X(k)\rangle\\
                                &\le&-\dfrac{\beta_1}{\alpha_2}\langle P(k),X(k)\rangle=-\dfrac{\beta_1}{\alpha_2}V(k,X(k)).
  \end{array}
\end{equation}
Therefore, $V$ is a global, uniform and exponential Lyapunov function for the system \eqref{eq:matdiffeqDT_LTV}, which proves that the system \eqref{eq:matdiffeqDT_LTV} is globally uniformly exponentially stable with rate $1-\beta_1/\alpha_2$.}\\

\black{\noindent\textbf{Proof that the statement \eqref{st:stabMatDT_LTV1} implies the statement \eqref{st:stabMatDT_LTV3}.} Assume that the system \eqref{eq:matdiffeqDT_LTV} is uniformly exponentially stable, then by definition there exist some $M\ge1$ and $\rho\in(0,1)$ such that
\begin{equation}
  ||\mathcal{D}_{k,k_0}(X)||_*\le M\rho^{k-k_0}||X||_*
\end{equation}
holds for all $k\ge k_0$, all $k_0\ge0$, and all $X\in\psd^n$. Now let $Q:\mathbb{Z}_{\ge0}\mapsto\pd^n$ be such that $\beta_1I\preceq Q(k)\preceq
\beta_2$ for some $\beta_1,\beta_2>0$ and define $\bar{P}:\mathbb{Z}_{\ge0}\mapsto\pd^n$ such that it verifies
  \begin{equation}
    \langle\bar{P}(k),X\rangle=\sum_{\tau=k}^\infty\langle Q(\tau),\mathcal{D}_{\tau,k}(X)\rangle
  \end{equation}
  for all $X\in\psd^n$ and all $k\ge0$, where $\mathcal{D}_{k,\tau}$ is defined in \eqref{eq:PHIDTMATLTV}, \eqref{eq:PHIDTMATLTV2}. We need to prove first that there exist scalars $\alpha_1,\alpha_2>0$ such that $\alpha_1I\preceq P(k)\preceq \alpha_2I$ for all $k\ge0$.

  Clearly, we have that
  \begin{equation}
    \begin{array}{rcl}
        \langle\bar{P}(k),X\rangle&=&\lim_{K\to\infty}\sum_{\tau=k}^K\langle Q(\tau),\mathcal{D}_{\tau,k}(X)\rangle\\
        &\le&\beta_2\lim_{K\to\infty}\sum_{\tau=k}^K\langle \mathcal{D}_{\tau,k}(X))\\
      &\le&\beta_2M\left(\lim_{K\to\infty}\sum_{\tau=k}^K\rho^{\tau-k}\right)||X||_*\\
      &=& \beta_2M\left(\lim_{K\to\infty}\dfrac{1-\rho^{K-k+1}}{1-\rho}\right)||X||_*\\
      &=&\dfrac{\beta_2M}{1-\rho}||X||_*=:\alpha_2||X||_*.
    \end{array}
  \end{equation}
  Additionally,
  \begin{equation}
    \langle\bar{P}(k),X\rangle\ge \langle Q(k),X\rangle\ge\beta_1||X||_*=:\alpha_1||X||_*.
  \end{equation}
  So, we have proven the existence of two scalars $\alpha_1,\alpha_2>0$ such that $\alpha_1I\preceq P(k)\preceq \alpha_2I$ for all $k\ge0$. \\

  Observing now that
  \begin{equation}
  \langle\bar{P}(k+1),\mathcal{D}_k(X)\rangle-\langle\bar{P}(k),X\rangle=-\langle Q(k),X\rangle
  \end{equation}
  yields
  \begin{equation}
    \langle \mathcal{D}_k^*(P(k+1))-P(k)+Q(k), X\rangle=0.
  \end{equation}
  Since the above expression holds for all $X\in\psd^n$, then this implies that the Lyapunov equation of statement \eqref{st:stabMatDT_LTV3} must hold with  $P=\bar{P}$ and the same $Q$. This proves the result.}
\end{proof}

\black{The following result is the specialization of the previous one to the LTI case where it is shown that it is possible to connect the stability properties of the system to some spectral properties of a matrix and an operator associated with the LTI version of the system \eqref{eq:matdiffeqDT_LTV}:
\begin{theorem}[LTI case]\label{th:stabMatDT_LTI}
  Assume that the system \eqref{eq:matdiffeqDT_LTV} is time-invariant. Then, the following statements are equivalent:
  \begin{enumerate}[(a)]
    \item The system \eqref{eq:matdiffeqDT_LTV} is exponentially stable.
    \item The conditions of the Theorem \ref{th:stabMatDT_LTV} in statements \eqref{st:stabMatDT_LTV3} and \eqref{st:stabMatDT_LTV4} hold with constant matrix-valued functions $P(t)\equiv P\in\pd^n$ and $Q(t)\equiv Q\in\sm^n$.
  \item The matrix
  \begin{equation}
    \mathcal{M}_{\mathcal{D}}:=F^{\T}\left(\sum_{i=0}^N J_i\otimes J_i\right)F
  \end{equation}
  is Schur stable.
   \item We have that $\sp(\mathcal{D})\subset\mathbb{D}$ where
  \begin{equation}
        \sp(\mathcal{D}):=\left\{\lambda\in\mathbb{C}: \mathcal{D}(X)=\lambda X\textnormal{ for some }X\in\mathbb{S}^{n},X\ne0\right\}.
  \end{equation}
   and $\mathcal{D}$ is defined as the time-invariant version of the operator \eqref{eq:operatorLdt}.
  \end{enumerate}
\end{theorem}
\begin{proof}
  The proof of the equivalence between the two first statements follow from Theorem \ref{th:stabMatDT_LTV} specialized to the LTI case.\\

   \noindent\textbf{Proof that statement (a) is equivalent to statement (d).} This follows from the fact that the LTI version of the system \eqref{eq:matdiffeqDT_LTV} is exponentially stable if and only if the eigenvalues associated with its dynamics are located inside the unit disc. Indeed, the solution to that LTI system is given by
  \begin{equation}
    X(k)=\sum_{i=1}^{n(n+1)/2}\lambda_i^k\langle X_0,X_i\rangle
  \end{equation}
  where $(\lambda_i,X_i)$ denotes the $i$-th pair of eigenvalues and eigenmatrices. This proves the equivalence.\\

\noindent\textbf{Proof that statement (c) is equivalent to statement (d).}  Using Kronecker calculus, Statement (d) is equivalent to saying that
\begin{equation}
    \left(\sum_{i=0}^NJ_i\otimes J_i-\lambda I_{n^2}\right)\vect(X)=0,\ X\ne0,X\in\mathbb{S}^n
\end{equation}
implies $|\lambda|<1$. The eigenvalues depend on the domain of the operator which is not the full space $\mathbb{R}^{n^2}$ here but the subspace associated with symmetric matrices $X$. This can be done by restricting the matrix  above to that subspace by projection. This can be achieved using the matrix $F$ as
\begin{equation}
  \left(\mathcal{M}_{\mathcal{D}}-\lambda I_{\frac{n(n+1)}{2}}\right)\overline\vect(X)=0,
\end{equation}
which is equivalent to statement (c).
\end{proof}}

\subsection{Stabilization}

\blue{Let us consider here the perturbation of the operator $\mathcal{D}_k$ given by
\begin{equation}\label{eq:operatorLdtu1}
  \begin{array}{rccl}
    \tilde{\mathcal{D}}_k:&\sm^n&\mapsto&\sm^n\\
    & X&\mapsto & \sum_{i=0}^N(J_i(k)+B_i(k)K_i(k)) X(J_i(k)+B_i(k)K_i(k))^\T
  \end{array}
\end{equation}
where the matrix-valued functions $J_i:\mathbb{Z}_{\ge0}\mapsto\mathbb{R}^{n\times n}$, $B_i:\mathbb{Z}_{\ge0}\mapsto\mathbb{R}^{n\times m_i}$, $K_i:\mathbb{Z}_{\ge0}\mapsto\mathbb{R}^{m_i\times n}$, $i=0,\ldots,N$, are assumed to be bounded. The objective is to formulate conditions for the existence of some $K_i$'s such that the above operator generates an uniformly exponentially stable state transition operator; i.e. $X(k+1)=\tilde{\mathcal{D}}_k(X(k))$ is globally uniformly exponentially stable. This is formulated in the following result:
\begin{theorem}\label{th:DT:stabz}
  The following statements are equivalent:
  \begin{enumerate}[(a)]
    \item There exist matrix-valued functions $K_i:\mathbb{Z}_{\ge0}\mapsto\mathbb{R}^{m_i\times n}$, $i=1,\ldots,$, such that the system $X(k+1)=\tilde{\mathcal{D}}_k(X(k))$ is uniformly exponentially stable.
    \item There exists a matrix-valued function $Q:\mathbb{Z}_{\ge0}\mapsto\pd^n$, $\alpha_1I\preceq Q(k)\preceq\alpha_2I$, such that the difference matrix inequality
    \begin{equation}
         \begin{bmatrix}
          -Q(k) & \row_{i=0}^N[(\mathscr{N}_i(k)^{\T}J_i(k)Q(k))^{\T}]\\
          \star & -\diag_{i=0}^N[\mathscr{N}_i(k)^{\T}Q(k+1)\mathscr{N}_i(k)]
        \end{bmatrix}\preceq -\alpha_3 I
    \end{equation}
    holds for some $\alpha_1,\alpha_2,\alpha_3>0$ and where $\mathscr{N}_i$ is such that $\mathscr{N}_i(k)^{\T}$ is a basis for the left null-space of $B_i(k)$; i.e. $\mathscr{N}_i(k)^{\T}B_i(k)=0$ and $\mathscr{N}_i(t)$ is of maximal rank.
    \item There exist a differential matrix-valued function $Q:\mathbb{Z}_{\ge0}\mapsto\pd^n$, $\alpha_1I\preceq Q(k)\preceq\alpha_2I$, and piecewise continuous matrix-valued functions $K_i:\mathbb{Z}_{\ge0}\mapsto\mathbb{R}^{m_i\times n}$ such that the differential matrix inequality
    \begin{equation}
\begin{bmatrix}
            -Q(k) & \row_{i=0}^N[(J_i(k)Q(k)+B_i(k)U_i(k))^{\T}]\\
             \star  & -I_N\otimes Q(k+1)
        \end{bmatrix}\preceq -\alpha_3 I
    \end{equation}
    holds for some $\alpha_1,\alpha_2,\alpha_3>0$. Moreover, suitable matrix-valued functions $K_i$'s are given by $K_i(k)=U_i(k)Q(k)^{-1}$.
  \end{enumerate}
\end{theorem}
\begin{proof}
  The equivalence between the first and the third statements follows from substituting the matrices of the operator \eqref{eq:operatorLdtu1} into the condition of the statement \eqref{st:stabMatDT_LTV3} of Theorem \eqref{th:stabMatDT_LTV}. Performing then a congruence transformation with respect to $Q(k):=P(k)^{-1}$, making the changes of variables  $U_i(k)=K_i(k)Q(k)$, and using a Schur complement leads to the result. The equivalence between the two last statements follows from the fact that the differential matrix inequality in the third statement can be rewritten as
  \begin{equation}
    \begin{bmatrix}
            -Q(k) & \row_{0=1}^N[(J_i(k)Q(k))^{\T}]\\
             \star  & -I_N\otimes Q(k+1)
        \end{bmatrix}+\He\left(\begin{bmatrix}
          0\\
          \diag_{i=0}^N B_i(k)
        \end{bmatrix}\col_{i=0}^N(U_i(k))\begin{bmatrix}
          I_n & 0_{n\times nN}
        \end{bmatrix}\right)\preceq -\alpha_3 I.
  \end{equation}
  Applying then the Projection Lemma \cite{Gahinet:94a} yields the condition in the second statement.
\end{proof}}

\subsection{Systems with delays}

\black{Interestingly, it is possible to extend some of the above results to the case of systems with delays, leading to an extension of the results in \citep{Tanaka:13b} to the discrete-time case with multiple delays. The LTI time-delay system version of the system \eqref{eq:matdiffeqDT_LTV} is given by
\begin{equation}\label{eq:matdiffeqDT_LTV:delay}
\begin{array}{rcl}
  X(k+1)&=&\sum_{i=0}^NJ_i X(k)J_i^\T+\sum_{i=1}^NH_i X(k-\tau_i)H_i^\T,\ k\ge 0\\
  X(k)&=&\Xi(k)\in\psd^n,\ k\in\{-\bar\tau,\ldots,0\}
\end{array}
\end{equation}
where $X(k)\in\psd^n$ is the state of the system, $\Xi:\{-\bar\tau,\ldots,0\}\mapsto\psd^n$ is the initial condition and the delays are such that $\tau_i\ge0$, $i=1,\ldots,N$, with $\bar\tau=\max_i\{\tau_i\}$. It is also convenient to define the operators $\mathcal{T}_i$ as
\begin{equation}
  \begin{array}{rccl}
    \mathcal{T}_i:&\sm^n&\mapsto&\sm^n\\
    & X&\mapsto & H_i XH_i^\T
  \end{array}
\end{equation}
and their adjoints $\mathcal{T}_i^*$ as
\begin{equation}
  \begin{array}{rccl}
    \mathcal{T}_i^*:&\sm^n&\mapsto&\sm^n\\
    & X&\mapsto & H_i^\T XH_i.
  \end{array}
\end{equation}
It can also easily be shown that the solution of the system \eqref{eq:matdiffeqDT_LTV:delay} remains confined in the positive semidefinite cone provided that the initial condition take values inside the positive definite cone.

The following result demonstrates that the same delay-independent stability property holds for the system \eqref{eq:matdiffeqDT_LTV:delay} as for linear time-invariant positive discrete-time systems \citep{Aleksandrov:14}:
\begin{theorem}\label{th:DT:delay}
The following statements are equivalent:
  \begin{enumerate}[(a)]
    \item\label{st:DT:delay1} The system \eqref{eq:matdiffeqDT_LTV:delay} is exponentially stable for any delay values $\tau_i\ge0$, $i=1,\ldots,N$.
    \item\label{st:DT:delay1b} The system \eqref{eq:matdiffeqDT_LTV:delay} with zero delays (i.e. $\tau_i=0$, $i=1,\ldots,N$) is exponentially stable.
    \item\label{st:DT:delay2} There exist matrices $P,Q_i\in\pd^n$, $i=1,\ldots,N$, such that the conditions
    \begin{equation}\label{eq:kdospdosadadkDT}
   \mathcal{D}^*(P)-P+\sum_{i=1}^NQ_i\prec0\ \textnormal{and }\mathcal{T}_i^*(P)-Q_i\preceq0,\ i=1,\ldots,N
\end{equation}
    hold.
    \item\label{st:DT:delay3} We have that $\sp\left(\mathcal{D}+\sum_{i=1}^N\mathcal{T}_i\right)\subset\mathbb{D}$ where
  \begin{equation}
        \sp\left(\mathcal{D}+\sum_{i=1}^N\mathcal{T}_i\right):=\left\{\lambda\in\mathbb{C}: \mathcal{D}(X)+\sum_{i=1}^N\mathcal{T}_i(X)=\lambda X\textnormal{ for some }X\in\mathbb{S}^{n},X\ne0\right\}.
  \end{equation}
  \end{enumerate}
\end{theorem}}
\black{\begin{proof}
We know from Theorem \ref{th:stabMatDT_LTI} that the statements \eqref{st:DT:delay1b} and \eqref{st:DT:delay3} are equivalent. It is also immediate to see that the statement \eqref{st:DT:delay1} implies the statement \eqref{st:DT:delay1b}. So, we need to prove the two remaining implications.\\

\noindent\textbf{Proof that the statement \eqref{st:DT:delay2} implies the statement \eqref{st:DT:delay1}.} Let us consider the following Lyapunov-Krasovskii functional
\begin{equation}\label{eq:LKF:DT}
  V(X_k)=\langle P,X_k(0)\rangle+\sum_{i=1}^N\sum_{j=-\tau_i}^{-1}\langle Q_i,X_k(j)\rangle
\end{equation}
where $X_k(s)=X(k+s)$, $s\in\{-\bar\tau,\ldots,0\}$. Clearly, we have that
\begin{equation}
  V(X_k)\ge\lmin(P)||X_k(0)||_*
\end{equation}
and
\begin{equation}
  V(X_k)\le(\lmax(P)+\lmax(Q))||X_k||_{d,*}
\end{equation}
where $||X_k||_{d,*}:=\max_{s\in\{-\bar\tau,\ldots,0\}}||X(k+s)||_{*}$.\\

\noindent The discrete-time derivative of the functional \eqref{eq:LKF:DT}  is given by
\begin{equation}
  \begin{array}{rcl}
    V(X_{k+1})-V(X_{k})   &=&   \langle P,X(k+1)-X(k)\rangle+\sum_{i=1}^N\langle Q_i,X(k)\rangle-\langle Q_i,X(k-\tau_i)\rangle\\
                                            &=&   \left\langle P,\mathcal{D}(X(k))+\sum_{i=1}^N\mathcal{T}_i(X(k-\tau_i))-X(k)\right\rangle+\sum_{i=1}^N\langle Q_i,X(k)\rangle-\langle Q_i,X(k-\tau_i)\rangle\\
                                            &=&   \left\langle \mathcal{D}^*(P)-P+\sum_{i=1}^NQ_i,X(k)\right\rangle+\sum_{i=1}^N\langle\mathcal{T}_i^*(P)-Q_i,X(k-\tau))\rangle.

  \end{array}
\end{equation}
Under the conditions of the statement \eqref{st:DT:delay2}, there exists a small enough $\eps>0$ such that $\textstyle\mathcal{C}^*(P)-P+\sum_{i=1}^NQ_i\preceq -\eps I$ which implies
\begin{equation}
  \begin{array}{rcl}
     V(X_{k+1})-V(X_{k})   &\le&   -\eps ||X_k(0)||_*
  \end{array}
\end{equation}
Therefore, the functional \eqref{eq:LKF:DT} is a discrete-time Lyapunov-Krasovskii functional \citep{Fridman:14} for the system \eqref{eq:matdiffeqDT_LTV:delay} and the system is exponentially stable regardless the value of the delays as the conditions do not depend on the values of the delays. This proves the implication.\\

\noindent\textbf{Proof that the statement \eqref{st:DT:delay1b} implies the statement \eqref{st:DT:delay2}.} Assume that the system \eqref{eq:matdiffeqDT_LTV:delay} with zero delays is exponentially stable, then the system
\begin{equation}
  X(k+1)=\mathcal{D}(X(k))+\sum_{i=1}^N\mathcal{T}_i(X(k))
\end{equation}
is exponentially stable. This implies that there exists a matrix $P\in\pd^n$ and an $\eps>0$ such that
\begin{equation}
  \mathcal{D}^*(P)-P+\sum_{i=1}^N\mathcal{T}_i^*(P)\preceq-2\eps I.
\end{equation}
Define now
\begin{equation}
  Q_i=\mathcal{T}_i^*(P)+\dfrac{\eps}{N}I_n\succ0
\end{equation}
for some $\eps>0$.  Substituting that in \eqref{eq:kdospdosadadkDT} yields
\begin{equation}
  \begin{array}{rcl}
     \mathcal{D}^*(P)-P+\sum_{i=1}^NQ_i&=&  \mathcal{D}^*(P)-P+\sum_{i=1}^N\mathcal{T}_i^*(P)\preceq-\eps I,\\
     \mathcal{T}_i^*(P)-Q_i&=&-\dfrac{\eps}{N}I_n\preceq0,
  \end{array}
\end{equation}
which implies that the conditions in statement \eqref{st:DT:delay2} hold, thereby proving the desired result. The proof is completed.
\end{proof}}

\subsection{Applications}

\blue{As in the continuous-time case, the simplicity and the versatility of the approach as a general tool for the analysis of (stochastic) discrete-time systems is illustrated through several examples, namely on the analysis of linear time-varying stochastic systems subject to difference types of stochastic finite-variance parameters, linear time-varying stochastic systems with delays, and time-varying systems with Markov jump parameters.}

\subsubsection{Stochastic LTV systems}

\black{We consider here the following stochastic discrete-time varying system
\begin{equation}\label{eq:syst:DT_general_LTV}
  \begin{array}{rcl}
    x(k+1) &=&A_0(k)x(k)+\sum_{i=1}^N\nu_i(k)A_i(k)x(k)\\
    x(k_0)&=&x_0
  \end{array}
\end{equation}
where $x(k),x_0\in\mathbb{R}^n$ are the state of the system and the initial condition, $k_0$ is the initial time. The random signals  $\nu_1(k),\ldots,\nu_N(k)\in\mathbb{R}$ are assumed to be independent of each other and from the state $x(k)$, and they are also assumed to be stationary. The matrix-valued functions $A_i$, $i=0,\ldots,N$, are assumed to be bounded for all $k$'s.

We consider here the following notion of stability for the system \eqref{eq:syst:DT_general_LTV}:
\begin{definition}\label{def;MSS2}
  The system \eqref{eq:syst:DT_general_LTV} is uniformly exponentially mean-square stable if there exist some scalars $\rho\in[0,1)$ and $\beta>1$ such that
  \begin{equation}
    \E[||x(k)||_2^2]\le\beta \rho^{k-k_0}\E[||x_0||_2^2]
  \end{equation}
  holds for all $k\ge k_0$, all $k_0\ge0$, and all $x_0\in\mathbb{R}^n$.
\end{definition}

As in the continuous-time case, this notion of stability is directly connected to the dynamics of the \blue{seconder-order moment} matrix associated with the system \eqref{eq:syst:DT_general_LTV}. This remark yields the following result:
\begin{theorem}
  Assume that the random signals $\nu_1(k),\ldots,\nu_N(k)$ have zero mean and unit variance. Then, the following statements are equivalent:
  \begin{enumerate}[(a)]
    \item The system \eqref{eq:syst:DT_general_LTV} is  uniformly exponentially mean-square stable.
    \item  There exist some matrix-valued functions $P,Q:\mathbb{Z}_{\ge0}\mapsto\mathbb{S}^n_{\succ0}$ such that
    \begin{equation}
    \begin{array}{rcccl}
      \alpha_1I&\preceq& P(k)&\preceq &\alpha_2I,\\
      \beta_1I&\preceq &Q(k)&\preceq &\beta_2I,
    \end{array}
    \end{equation}
    and such that Lyapunov difference equation
    \begin{equation}
     \sum_{i=0}^NA_i(k)^{\T}P(k+1)A_i(k)-P(k)=-Q(k)
  \end{equation}
  hold for some positive constants $\alpha_1,\alpha_2,\beta_1,\beta_2$ and for all $k\ge0$.
  \end{enumerate}
\end{theorem}
\begin{proof}
  Since the $\nu_i$'s have zero mean and unit variance, then the \blue{seconder-order moment} matrix associated with the system \eqref{eq:syst:DT_general_LTV} is described by the matrix difference equation
\begin{equation}
  X(k+1)  =  \sum_{i=0}^NA_i(k)X(k)A_i(k)^{\T}
\end{equation}
together with $X(k_0)=\E[x(k_0)x(k_0)^{\T}]$. Applying then Theorem \ref{th:stabMatDT_LTV} yields the result.
\end{proof}

\begin{theorem}
   Assume that the random signals $\nu_1(k),\ldots,\nu_N(k)$ follow a Bernoulli distribution with  $\P(\nu_i(k)=1)=p_i$ and $\P(\nu_i(k)=1)=1-p_i$, $i=1,\ldots,N$. Then, the following statements are equivalent:
  \begin{enumerate}[(a)]
    \item The system \eqref{eq:syst:DT_general_LTV} is  uniformly exponentially mean-square stable.
    \item  There exist some matrix-valued functions $P,Q:\mathbb{Z}_{\ge0}\mapsto\mathbb{S}^n_{\succ0}$ such that
    \begin{equation}
    \begin{array}{rcccl}
      \alpha_1I&\preceq& P(k)&\preceq &\alpha_2I,\\
      \beta_1I&\preceq &Q(k)&\preceq &\beta_2I,
    \end{array}
    \end{equation}
    and such that Lyapunov difference equation
    \begin{equation}
    \sum_{i=0}^N\bar{A}_i(k)^{\T}P(k+1)\bar{A}_i(k)-P(k)+Q(k)=0
  \end{equation}
  hold for some positive constants $\alpha_1,\alpha_2,\beta_1,\beta_2$ and for all $k\ge0$ where
  \begin{equation}
    \begin{array}{rcl}
        \bar{A}_0(k)&=&A_0(k)+\sum_{i=1}^Np_iA_i(k),\\
        \bar{A}_i(k)&=&(p_i(1-p_i))^{1/2}A_i(k),\ i=1,\ldots,N.
    \end{array}
  \end{equation}
  \end{enumerate}
\end{theorem}
\begin{proof}
  First note that we have $\E[\nu_i]=p_i$, $\E[\nu_i^2]=p_i$, $\E[\nu_i\nu_j]=p_ip_j$, and $\textnormal{Var}(\nu_i) = p_i(1-p_i)$. The \blue{seconder-order moment} associated with the system \eqref{eq:syst:DT_general_LTV} is given by
\begin{equation}
  \begin{array}{rcl}
    X(k+1)=\bar{A}_0(k)X(k)\bar{A}_0(k)^{\T}+\sum_{i=1}^N\bar{A}_i(k)X(k)\bar{A}_i(k)^{\T}
  \end{array}
\end{equation}
where the matrices are defined in the result. Applying then Theorem \ref{th:stabMatDT_LTV} yields the result.
%
  %
\end{proof}}

\subsubsection{A class of stochastic linear system with delays}

\black{We consider here the following stochastic discrete-time varying system
\begin{equation}\label{eq:delaystoch:DT}
  \begin{array}{rcl}
    x(k+1) &=&A_0x(k)+\sum_{i=1}^N\nu_i(k)A_ix(k-\tau_i)\\
    x(s)&=&\phi(s),\ s\in\{-\bar\tau,\ldots,0\}
  \end{array}
\end{equation}
where $x(k)\in\mathbb{R}^n$ is the state of the system, $\phi:\{-\bar \tau,\ldots,0\}\mapsto\mathbb{R}^n$ is the initial condition and the delays $\tau_i$ are such that $\tau_i\ge0$, $i=1,\ldots,N$ with $\textstyle\bar\tau:=\max_i\{\tau_i\}$. The zero mean and unit variance random signals $\nu_1(k),\ldots,\nu_N(k)$ are assumed to be independent of each other and from the state $x(k)$.

Since the system now involves some time-delays, we need to adapt the notion of mean-square exponential stability to this setup as follows:
\begin{definition}
  The system \eqref{eq:delaystoch:DT} is mean-square exponentially stable if there exist some scalars $\rho\in(0,1)$ and $\beta>1$ such that
  \begin{equation}
    \E[||x(k)||_2^2]\le\beta \rho^{k}\E\left[\max_{s\in\{-\bar \tau,\ldots,0\}}||\phi(s)||_2^2\right]
  \end{equation}
  holds for all $k\ge0$ and all initial condition $\phi\in C(\{-\bar \tau,\ldots,,0\},\mathbb{R}^n)$.
\end{definition}

We have the following result:
\begin{theorem}
The following statements are equivalent:
\begin{enumerate}[(a)]
    \item The system \eqref{eq:delaystoch:DT} is mean-square  exponentially stable for all $\tau_i\ge0$, $i=1,\ldots,N$.
    \item The system  \eqref{eq:delaystoch:DT}  with zero-delays (i.e. with $\tau_i=0$ for all $i=1,\ldots,N$) is mean-square exponentially stable.
\end{enumerate}
\end{theorem}
\begin{proof}
     The \blue{seconder-order moment} matrix $X(k)$ associated with the system \eqref{eq:delaystoch:DT} is governed by
\begin{equation}
  X(k+1)=A_0(k)X(k)A_0(k)^{\T}+\sum_{i=1}^NA_i(k)X(k-\tau_i)A_i(k)^{\T}
\end{equation}
and $X(0)=\E[x(0)x(0)^{\T}]$. Applying now Theorem \ref{th:DT:delay} yields the result.
\end{proof}}

\subsubsection{Time-varying Markov jump linear systems}

\black{We consider here the following class of systems subject to Markovian jumps
\begin{equation}\label{eq:DT:Markov}
\begin{array}{rcl}
  x(k+1)&=&A_{\sigma(k)}(k)x(k)+\sum_{i=1}^NB_{\sigma(k),i}(t)x(k)\nu_i(k),\ k\ge k_0\vspace{-3mm}\\
  x(k_0)&=&x_0,\\
  \sigma(k_0)&=&\sigma_0
\end{array}
\end{equation}
where $x(k),x_0\in\mathbb{R}^n$ are the state of the system and the initial condition, $k_0$ is the initial time. The random signals  $\nu_1(k),\ldots,\nu_N(k)\in\mathbb{R}$ are assumed to be independent of each other and from the state $x(k)$ and the switching signal $\sigma(k)$, and are also assumed to be stationary. The switching signal $\sigma:\mathbb{Z}_{\ge t_0}\mapsto\{1,\ldots,M\}$, with initial value $\sigma_0$, evolves according to the finite Markov chain defined by
\begin{equation}\label{eq:DT:Markov2}
  \P(\sigma(k+1)=j|\sigma(k)=i)=\pi_{ij}
\end{equation}
where $\pi_{ij}\ge0$ for all $i=1,\ldots,M$, we have that $\textstyle\sum_{j=1}^M\pi_{ij}=1$. The matrix-valued functions $A_i$, $i=0,\ldots,M$, and $B_{i,j}$, $i=1,\ldots,M$, $j=1,\ldots,N$, are assumed to be bounded for all $k$'s.

\bblue{We consider here the concept of stability called \emph{uniformly exponential mean-square stability in conditioning} as defined in \citep{Dragan:02,Dragan:13}:
\begin{definition}\label{def:MSS2b}
  The system \eqref{eq:DT:Markov}\eqref{eq:DT:Markov2} is uniformly exponentially mean-square stable in conditioning if there exist some scalars $\rho\in[0,1)$ and $\beta\ge1$ such that
    \begin{equation}
    \E[||x(k)||_2^2|\sigma(k_0)]\le\beta \rho^{k-k_0}\E[||x_0||_2^2]
  \end{equation}
  holds for all $k\ge k_0$, all $k_0\ge0$, and all $x_0\in\mathbb{R}^n$ and all initial probability distribution $\pi(t_0)$ of the Markov process defined by \eqref{eq:DT:Markov2}.
\end{definition}

This leads to the following result\footnote{It was brought to the attention of the author after publication that this result was previously published in \citep{Dragan:06,Dragan:10}.} which can be seen as a generalization of the results in\citep{Costa:05}:}
\begin{theorem}\label{th:Markov:DT}
  \bblue{Assume that the matrix $\Pi=[\pi_{ij}]$ has no zero column.} Then, the system \eqref{eq:DT:Markov}-\eqref{eq:DT:Markov2} is \bblue{uniformly exponentially mean-square stable in conditioning} if and only if there exist matrix-valued functions $P_i,Q_i:\mathbb{Z}_{\ge0}\mapsto\pd^n$, $i=1,\ldots,M$, such that
  \begin{equation}
    \begin{array}{rcccl}
      \alpha_1I&\preceq&P_i(k)&\preceq&\alpha_2I,\\
      \beta_1I&\preceq&Q_i(k)&\preceq&\beta_2I,
    \end{array}
  \end{equation}
  and the difference equations
\begin{equation}\label{eq:dl;skdl;kasdk;ldk;sdk;ak:DT}
  A_i(k)^{\T}\left(\sum_{j=1}^M\pi_{ij}P_j(k+1)\right)A_i(k)+\sum_{\ell=1}^N B_{i,\ell}(k)^{\T}\left(\sum_{j=1}^M\pi_{ij}P_j(k+1)\right)B_{i,\ell}(k)-P_i(k)=-Q_i(k)
\end{equation}
hold for all $i=1,\ldots,M$.
\end{theorem}
\begin{proof}
  Let $\mathds{1}_i(\sigma(k))$ be the indicator function of the state $i$ defined as $\mathds{1}_i(\sigma(k))=1$  if $\sigma(k)=i$, and 0 otherwise. Then, defining $X_i(k)=\E[x(k)x(k)^{\T}\mathds{1}_i(\sigma(k))]$, we get that
\begin{equation}
  X_i(k+1)=\sum_{j=1}^M\pi_{ji}\left(A_j(k)X_j(k)A_j(k)^{\T}+\sum_{\ell=1}^NB_{j,\ell}(k)X_j(k)B_{j,\ell}(k)^{\T}\right)
\end{equation}
for all $i=1,\ldots,M$. Let $E_i$ be the matrix composed of the columns $\{e_1,\ldots,e_M\}$ in cyclic ascending order where $e_1$ is the $i$-th column and define the matrices $A^\pi(k)$ and $B^\pi_\ell(k)$ as
\begin{equation}
  \begin{array}{rcl}
  \ [A^\pi(k)]_{ij}&=&\sqrt{\pi_{ji}}A_{j}(k),\\
  \ [B_\ell^\pi(k)]_{ij}&=&\sqrt{\pi_{ji}}B_{j,\ell}(k).
  \end{array}
\end{equation}
Then, we have that
\begin{equation}
  \diag_{i=1}^M\left(\sum_{j=1}^M\pi_{ji}A_{j}(k)X_j(k)A_{j}(k)^{\T}\right)=\sum_{i=1}^M\bar{A}_{i}(k)X(k)\bar{A}_{i}(k)^{\T}
\end{equation}
and
\begin{equation}
  \diag_{i=1}^M\left(\sum_{j=1}^M\pi_{ji}B_{j,\ell}(k)X_j(k)B_{j,\ell}(k)^{\T}\right)=\sum_{i=1}^M\bar{B}_{i,\ell}(k)X(k)\bar{B}_{i,\ell}(k)^{\T}
\end{equation}
where $X(t):=\textstyle\diag_{i=1}^M\{X_i(t)\}$, $\bar{A}_{i}(k)=A^\pi(k)\odot E_i$ and $\bar{B}_{i,\ell}=B^\pi_\ell\odot E_i$. Then, one can rewrite the above $M$ coupled matrix-valued differential equations as
\begin{equation}
  X(k+1)=\sum_{i=1}^N\left(\bar{A}_i(k)X(k)\bar{A}_i(k)^{\T}+\sum_{\ell=1}^N\bar{B}_{i,\ell}(k)X(k)\bar{B}_{i,\ell}(k)^{\T}\right).
\end{equation}
Using Theorem \ref{th:stabMatDT_LTV}, we obtain the following stability condition for the above dynamical system
\begin{equation}\label{eq:kdsl;kd;sakl;dkas;dk;salkd;ladkl;sakd;sak;d}
  \sum_{i=1}^M\left(\bar{A}_i(k)^{\T}P(k+1)\bar{A}_i(k)+\sum_{\ell=1}^N\bar{B}_{i,\ell}(k)^{\T}P(k+1)\bar{B}_{i,\ell}(k)\right)-P(k)=-Q(k)
\end{equation}
for some matrix-valued functions $P,Q$ satisfying the conditions in Theorem \ref{th:stabMatDT_LTV}. As the state of this system has a block-diagonal structure, the state space is actually $\mathbb{S}^n_{\succeq0}\times\ldots\mathbb{S}^n_{\succeq0}$ ($M$ times), and not the full space $\mathbb{S}^{nM}_{\succeq0}$. Therefore, we can consider, without introducing any conservatism, block-diagonal matrix-valued functions of the form $\textstyle P(k)=\diag_{i=1}^M(P_i(k))$ and $\textstyle Q(k)=\diag_{i=1}^M(Q_i(k))$ where $P_i(k),Q_i(k)\in\mathbb{S}_{\succ0}^n$, $i=1,\ldots,M$, all satisfy the boundedness conditions of Theorem \ref{th:stabMatDT_LTV}. Using this structure, we obtain
\begin{equation}
  \sum_{i=1}^M\bar{A}_i(k)^{\T}P(k+1)\bar{A}_i(k)=\diag_{i=1}^M\left(A_i(k)^{\T}\left(\sum_{j=1}^M\pi_{ij}P_j(k+1)\right)A_i(k)\right)
\end{equation}
and
\begin{equation}
  \sum_{i=1}^M\bar{B}_{i,\ell}(k)^{\T}P(k+1)\bar{B}_{i,\ell}(k)=\diag_{i=1}^M\left(B_{i,\ell}(k)^{\T}\left(\sum_{j=1}^M\pi_{ij}P_j(k+1)\right)B_{i,\ell}(k)\right).
\end{equation}
Considering, finally, the above expressions and expanding the other terms in \eqref{eq:kdsl;kd;sakl;dkas;dk;salkd;ladkl;sakd;sak;d} yields the conditions \eqref{eq:dl;skdl;kasdk;ldk;sdk;ak:DT}. The proof is now completed.
\end{proof}}

\section{Linear Matrix-Valued Symmetric Impulsive Systems}\label{sec:Hybrid}

\blue{The objective of this section is to merge the results obtained in the case of continuous-time and discrete-time all together to obtain analogous stability conditions in the hybrid setting. Linear matrix-valued symmetric impulsive systems that evolves on the cone of positive semidefinite matrices are first introduced. A necessary and sufficient condition for their uniform exponential stability is then obtained in Theorem \ref{th:Main:Imp} in terms of a matrix-differential equation or a matrix-differential inequality. Relaxed stability conditions, analogous to those considered in \citep{Goebel:12,Briat:19:Linf} are then obtained in Theorem \ref{th:Main:Imp:persflow} and Theorem \ref{th:Main:Imp:persjump}. Convex stabilization conditions are obtained in Theorem \ref{th:Main:Imp:stabz} and take the form of differential-different linear matrix inequalities. The simplicity and the versatility of the approach as a general tool for the analysis of (stochastic) continuous-time systems is illustrated through several examples, namely on the analysis of stochastic LTV systems with deterministic impulses times \citep{Michel:08,Goebel:12,Briat:15i}, stochastic LTV switched systems with deterministic switchings \citep{Michel:08,Goebel:12,Briat:14f,Shaked:14,Briat:15i}, sampled-data stochastic LTV systems  \citep{Michel:08,Goebel:12,Briat:15i}, and stochastic LTV impulsive systems with stochastic impulses and switchings \citep{Souza:21}.}

\subsection{Preliminaries}

We consider here the following class of matrix-valued impulsive dynamical systems
\begin{equation}\label{eq:matdiffeqHybridLTV}
\begin{array}{rcl}
  \dot{X}(t)&=&A_0(t) X(t)+X(t)A_0(t)^\T+\sum_{i=1}^NA_i(t) X(t)A_i(t)^\T+\mu X(t),t\ne t_k,t\ge t^0\\
  X(t^+)&=&\sum_{i=0}^NJ_i(k) X(t)J_i(k)^\T,\ t=t_k, t\ge t^0\\
  X(t^0)&=&X_0\in\psd^n
\end{array}
\end{equation}
where $X(t),X_0\in\psd^n$ is the state of the system and the initial condition, and $t^0\ge0$ is the initial time. The matrix-valued functions $A_i:\mathbb{R}_{\ge0}\mapsto\mathbb{R}^{n\times n}$, $i=1,\ldots,N$, are assumed to be piecewise continuous and uniformly bounded whereas the matrix-valued functions $J:\mathbb{Z}_{\ge0}\mapsto\mathbb{R}^{n\times n}$ are assumed to be uniformly bounded. The sequence of impulse instants $\{t_k\}_{\ge1}$ is assumed to be strictly increasing and to grow unboundedly; i.e. $t_{k+1}>t_k$, $k\ge1$ and $t_k\to\infty$ as $k\to\infty$. Under those assumptions on the data of the system, the impulse sequence and the linearity of the system, we know that there exists a unique complete (i.e. defined for all times $t\ge t^0$) solution, which we denote by $X(t,t^0,X_0)$. This solution can expressed in terms of a state-transition operator corresponding to the generators $\mathcal{C}_t$ and $\mathcal{D}_k$ defined in the previous sections, that is we have that
\begin{equation}
  X(t)=S(t,s)X(s),\ t\ge s\ge t^0
\end{equation}
where
\begin{equation}
  \dfrac{\partial}{\partial t}S(t,s)X=\mathcal{C}_t(S(t,s)X),\ S(s,s)=I,\ t_{k+1}\ge t\ge s>t_k,\ k\ge1
\end{equation}
and
\begin{equation}
  S(t,t_k)X=S(t,t_k^+)\circ\mathcal{D}_k(X),\ t\in(t_k,t_{k+1}],\ k\ge1.
\end{equation}

We will consider in this section the following notion of stability for the system \eqref{eq:matdiffeqHybridLTV}:
\begin{definition} \label{def:HybridsystemMat}
The zero solution of the hybrid system \eqref{eq:matdiffeqHybridLTV} is \textbf{globally uniformly exponentially stable with hybrid rate $(\alpha,\rho)$} if there exist some constants $\beta\ge1$, $\alpha>0$ and $\rho\in(0,1)$ such that
  \begin{equation}
         ||X(t,t^0,X_0)||\le\beta\rho^{\kappa(t,t^0)}e^{-\alpha(t-t^0)}||X_0||
  \end{equation}
  for all $t\ge t^0$, all $t^0\in\Rnn$, and all $X_0\in\psd^n$ where $\kappa(t,t^0)$ denotes the number of jumps in the interval $[t^0,t]$.
\end{definition}

It is interesting to note that this stability notion may be relaxed to the cases $\alpha>0$, $\rho=1$, and $\alpha=0$, $\rho\in(0,1)$ which consists to the case of persistent flowing and persistent jumping, respectively (see e.g. \citep{Goebel:12}) under some additional conditions on the impulse time sequences. Indeed, persistent flowing requires that the time spent by the system flowing is infinite while persistent jumping requires that all the impulses arrive in finite time; i.e. $t_{k+1}-t_k$ is finite for all $k\ge0$, $t_0=t^0$. This will be further discussed in the next section.

\subsection{Main stability result}

\black{We have the following result:
\begin{theorem}\label{th:generalImpulsiveSystems:jdksjdlasdjakljdlsa}\label{th:Main:Imp}
    The following statements are equivalent:
  \begin{enumerate}[(a)]
    %
    \item The matrix-valued linear dynamical system \eqref{eq:matdiffeqHybridLTV} is globally uniformly exponentially stable with hybrid rate $(\alpha,\rho)$ for some $\alpha>0$ and $\rho\in(0,1)$.
    %
    \item There exist a piecewise differentiable matrix-valued function $P:\mathbb{R}_{\ge0}\mapsto\mathbb{S}^n_{\succ0}$, a piecewise continuous matrix-valued function $Q:\mathbb{R}_{\ge0}\mapsto\mathbb{S}_{\succ0}^n$, and a matrix-valued function $R:\mathbb{Z}_{\ge0}\mapsto\mathbb{S}_{\succ0}^n$ such that
    \begin{equation}
    \begin{array}{rcccl}
      \alpha_1I&\preceq& P(t)&\preceq &\alpha_2I,\\
      \beta_1I&\preceq &Q(t)&\preceq &\beta_2I,\\
      \gamma_1I&\preceq &R(k)&\preceq &\gamma_2I,
    \end{array}
    \end{equation}
        such that
    \begin{equation}\label{eq:SecondOrderMomentImpulsiveLyapunov1}
      \begin{array}{rclcl}
        \dot{P}(t)+\mathcal{C}_t(P(t))&=&-Q(t),&& t\ne t_k,\ t\ge t^0,\ k\ge1\\
        \mathcal{D}_k(P(t_k^+))-P(t_k)&=&-R(k),&& t_k>t^0,\ k\ge1
      \end{array}
    \end{equation}
    hold for some scalars $\alpha_1,\alpha_2,\beta_1,\beta_2,\gamma_1,\gamma_2>0$.
     \item There exists a piecewise differentiable matrix-valued function $P:\mathbb{R}_{\ge0}\mapsto\mathbb{S}^n_{\succ0}$ such that
     \begin{equation}
       \alpha_1 I\preceq P(t)\preceq \alpha_2I,
     \end{equation}
     and
    \begin{equation}\label{eq:SecondOrderMomentImpulsiveLyapunov2}
      \begin{array}{rclcl}
         \dot{P}(t)+\mathcal{C}_t(P(t))&\preceq&-\alpha_3I,&& t\ne t_k,\ t\ge t^0,\ k\ge1\\
        \mathcal{D}_k(P(t_k^+))-P(t_k)&\preceq&-\alpha_4I,&& t_k>t^0,\ k\ge1
      \end{array}
    \end{equation}
    hold for some scalars $\alpha_1,\alpha_2,\alpha_3,\alpha_4,\beta_1,\beta_2,\gamma_1,\gamma_2>0$.
  \end{enumerate}
\end{theorem}
\begin{proof}
The statements (b) and (c) are readily seen to be equivalent. Therefore, we need to focus on the connection with the statement (a).\\

  \noindent\textbf{Proof that the statement (b) implies the statement (a).} Assume that the condition of the statement (b) hold and define the function $V(t,X)=\langle P(t),X\rangle$. From the conditions in statement (b), this function is $\psd^n$-copositive definite, radially unbounded and decrescent. Computing now its derivative along the flow of the system \eqref{eq:matdiffeqHybridLTV} yields
  \begin{equation}
    \begin{array}{rcl}
      \dot{V}(t,X(t))   &=&     \langle \dot{P}(t),X(t)\rangle+\langle P(t),\mathcal{C}_t(X(t))\rangle\\
                                    &=&     \langle \dot{P}(t)+\mathcal{C}^*_t(P(t)), X(t)\rangle\\
                                    &=&     -\langle Q(t),X(t)\rangle\\
                                    &\le& -\dfrac{\beta_1}{\alpha_2}V(t,X(t)).
    \end{array}
  \end{equation}
  Similarly, evaluating the discrete-time derivatives at jump times yields
  \begin{equation}
        \begin{array}{rcl}
      V(t_k^+,X(t_k^+))-V(t_k,X(t_k))   &=&     \langle P(t_k^+),\mathcal{D}_k(X(t_k))\rangle-\langle P(t_k),X(t_k)\rangle\\
                                                                    &=&         \langle \mathcal{D}_k^*(P(t_k^+))-P(t_k),X(t_k)\rangle\\
                                                                    &=&         - \langle R(k),X(t_k)\rangle\\
                                                                    &\le&-\dfrac{\gamma_1}{\alpha_2}V(t_k,X(t_k)).
    \end{array}
  \end{equation}
  Letting $\alpha:=\dfrac{\beta_1}{\alpha_2}>0$ and $\rho=1-\dfrac{\gamma_1}{\alpha_2}\in(0,1)$, we get that
  \begin{equation}
    \begin{array}{rcl}
      \dot{V}(t,X(t))&\le&-\alpha V(t,X(t)),\  t\ne t_k,\ t\ge t^0,\ k\ge1\\
      V(t_k^+,X(t_k^+))&\le&\rho V(t_k,X(t_k)),\ t_k>t^0,\ k\ge1.
    \end{array}
  \end{equation}
  Integrating this expression yields
  \begin{equation}
    V(t,X(t))\le V(0,X_0)e^{-\alpha(t-t_0)}\rho^{\kappa(t,t_0)}
  \end{equation}
   which implies
   \begin{equation}
     ||X(t,t_0,X_0)||\le\dfrac{\alpha_2}{\alpha_1}e^{-\alpha(t-t_0)}\rho^{\kappa(t,t_0)}||X_0||
   \end{equation}
   which proves the uniform hybrid exponential convergence. This proves the desired implication.\\

  \noindent\textbf{Proof that the statement (a) implies the statement (c).} Assume that the conditions of statement (a) hold and define $\bar P(t)$ as
  \begin{equation}\label{eq:explicitP(t)}
     \langle \bar P(t),X(t)\rangle=\int_t^\infty \langle Q(s),X(s)\rangle\ds+\sum_{t\le t_i}\langle R(i),X(t_i)\rangle
  \end{equation}
  We need to show first that $||P(t)||$ is uniformly bounded and bounded away from zero. From the definition of uniform hybrid exponential stability, i.e. Definition \ref{def:HybridsystemMat}, we have that
  \begin{equation}
  \begin{array}{rcl}
    \langle \bar P(t),X(t)\rangle&=&\lim_{\tau\to\infty}\left[\int_t^\tau \langle Q(s),S(s,t)X(t)\rangle\ds+\sum_{\tau>t_i\ge t}\langle R(i),S(t_i,t)X(t)\rangle \right]\\
    &\le& M||X(t)||_*\lim_{\tau\to\infty} \left(\beta_{X,2}\int_t^\tau e^{-\alpha(s-t)}\rho^{\kappa(s,t)}\ds+\beta_{Y,2}\sum_{\tau>t_i\ge t}\rho^{\kappa(t_i,t)}e^{-\alpha(t_i-t)}\right)\\
    &\le&M||X(t)||_*\lim_{\tau\to\infty}\left(\beta_{X,2}\int_t^\tau e^{-\alpha(s-t)}\ds+\beta_{Y,2}\sum_{\tau>t_i\ge t}\rho^{\kappa(t_i,t)}\right)\\
    &=&M||X(t)||_*\left(\dfrac{\beta_{X,2}}{\alpha}+\dfrac{\beta_{Y,2}}{1-\rho}\right).
  \end{array}
  \end{equation}
  Hence, there exists an $\alpha_2>0$ such that $\bar P(t)\preceq\alpha_2I$.\\

\noindent Let us now prove that there exists an $\alpha_1>0$ such that $\bar P(t)\succeq\alpha_1I$. To this aim, we consider
\begin{equation}
\begin{array}{rcl}
  \dfrac{\d}{\d s}\langle Q(s),X(s)\rangle&=&\langle \dot{Q}(s),X(s)\rangle +\langle Q(s),\mathcal{C}_s(X(s))\rangle\\
  &=& \langle \dot{Q}(s),X(s)\rangle +\langle \mathcal{C}_s^*(Q(s)),X(s)\rangle\\
  &=& \langle \dot{Q}(s)+\mathcal{C}_s^*(Q(s)),X(s)\rangle
\end{array}
\end{equation}
where we have assumed that $Q(s)$ is piecewise differentiable with bounded derivative. This is not a restrictive assumption as it can be chosen as a constant without loss of generality. Now, using the fact that $Q(s)$ is bounded with bounded derivative, then there exists a large enough $c>0$ such that
\begin{equation}
\langle \dot{Q}(s)+\mathcal{C}_s^*(Q(s)),X(s)\rangle \ge-c\langle  Q(s),X(s)\rangle.
\end{equation}
This implies that for some sufficiently large $c>0$, we have that
\begin{equation}
\dfrac{\d}{\d s}\langle Q(s),X(s)\rangle \ge-c\langle  Q(s),X(s)\rangle
\end{equation}
for all $X(s)\in\psd^n$ and for all $t_k< s\le t\le t_{k+1}$, $k\in\mathbb{Z}_{\ge0}$.\\

\noindent Now assume that $t_{k-1}<t\le t_{k}$ and observe that
\begin{equation}
\begin{array}{rcccl}
   \int_{t}^{t_k}\dfrac{\d}{\d s}\langle Q(s),X(s)\rangle\ds&=&\langle Q(t_k),X(t_k)\rangle - \langle Q(t),X(t)\rangle&\ge & -c\int_{t}^{t_k}\langle Q(s),X(s)\rangle \ds\\
   \int_{t_i}^{t_{i+1}}\dfrac{\d}{\d s}\langle Q(s),X(s)\rangle\ds&=&\langle Q(t_{i+1}),X(t_{i+1})\rangle-\langle Q(t_{i}^+),X(t_{i}^+)\rangle&\ge & -c\int_{t_i}^{t_{i+1}}\langle Q(s),X(s)\rangle \ds.
\end{array}
\end{equation}
Summing all the above terms gives
\begin{equation}\label{eq:kds;kdksl;kad;ls;ddsdsddsdsk}
  -\langle Q(t),X(t)\rangle+\sum_{i=k}^\infty\left(\langle Q(t_i),X(t_{i})\rangle-\langle Q(t_i^+),X(t_{i}^+)\rangle\right)  \ge -c\int_{t}^{\infty}\langle Q(s),X(s)\rangle \ds.
\end{equation}
Adding $-c\sum_{t_i\ge t}\langle R(i),X(t_i)\rangle=-c\sum_{i=k}^\infty\langle R(i),X(t_i)\rangle$ on both sides yields
\begin{equation}
   -\langle Q(t),X(t)\rangle+\sum_{i=k}^\infty\left(\langle Q(t_i),X(t_{i})\rangle-\langle Q(t_i^+),X(t_{i}^+)\rangle\right)-c\sum_{i=k}^\infty\langle R(i),X(t_i)\rangle  \ge -c\langle \bar P(t),X(t)\rangle.
\end{equation}
where we have used the fact that $t_{k-1}<t\le t_{k}$. Noting that $X(t_{i}^+)=\mathcal{D}_i(X(t_{i}))$, then we get that
\begin{equation}
\begin{array}{rcl}
  \sum_{i=k}^\infty\left(\langle Q(t_i),X(t_{i})\rangle-\langle Q(t_i^+),X(t_{i}^+)\rangle-c\langle R(i),X(t_i)\rangle\right)&=& \sum_{i=k}^\infty\left(\langle Q(t_i)-cR(i),X(t_{i})\rangle-\langle Q(t_i^+),\mathcal{D}_i(X(t_{i}))\rangle\right)\\
  &=& \sum_{i=k}^\infty  \langle Q(t_i)-cR(i)-\mathcal{D}_i^*(Q(t_{i}^+)),X(t_{i})\rangle
\end{array}
\end{equation}
Therefore, and since the matrix-valued functions $J_j$ are uniformly bounded, then there exists a large enough $c>0$ (which stays compatible with out previous choice for $c$) such that $Q(t_i)-cR(i)-\mathcal{D}_i^*(Q(t_{i}^+))\prec0$. Therefore, with such a $c>0$, we get that
\begin{equation}
  -\langle Q(t),X(t)\ge -c\langle \bar P(t),X(t)\rangle
\end{equation}
and, as a result, $\bar P(t)\succeq Q(t)/c\succeq \beta_1I/c=:\alpha_1I$. This proves the second inequality for $\bar{P}(t)$.\\

\noindent We now prove that $\bar P(t)$ in \eqref{eq:explicitP(t)} verifies the equalities in \eqref{eq:SecondOrderMomentImpulsiveLyapunov1}. Computing the derivative of $\bar P(t)$ with respect to time yields
    \begin{equation}
    \begin{array}{rcl}
      \dfrac{\d}{\dt}\langle \bar P(t),X(t)\rangle   &=& \langle \dot{\bar{P}}(t),X(t)\rangle+\langle \bar P(t),\dot X(t)\rangle  \\
                                                                                &=& \langle \dot{\bar{P}}(t),X(t)\rangle+\langle \bar P(t),\mathcal{C}_t(X(t))\rangle  \\
                                                                                &=& \langle \dot{\bar{P}}(t),X(t)\rangle+\langle \mathcal{C}_t^*(\bar P(t)),X(t)\rangle  \\
                                                                                &=& -\langle Q(t),X(t)\rangle.
    \end{array}
  \end{equation}
  Therefore, $\langle \dot{\bar{P}}(t)+\mathcal{C}_t^*(\bar P(t))+Q(t),X(t)\rangle=0$ and \eqref{eq:explicitP(t)} verifies the first inequality in \eqref{eq:SecondOrderMomentImpulsiveLyapunov1} with $P=\bar P$. Similarly, evaluating $\bar P(t)$ at $t_k$ and $t_k^+$ yields
  \begin{equation}
    \begin{array}{rcl}
    \langle \bar P(t_k^+),X(t_k^+)\rangle&=&\sum_{i=k+1}^\infty\langle R(i),X(t_i)\rangle+\int_{t_k}^{\infty}\langle Q(s), X(s)\rangle\ds\\
    \langle \bar P(t_k),X(t_k)\rangle     &=&    \langle \bar P(t_k),X(t_k)\rangle\\
                                    &=& \sum_{i=k}^\infty \langle R(i),X(t_i)\rangle + \int_{t_k}^{\infty}\langle Q(s), X(s)\rangle\ds\\
                                    &=&   \langle R(k),X(t_k)\rangle + \sum_{i=k+1}^\infty \langle R(i),X(t_i)\rangle + \int_{t_k}^{\infty}\langle Q(s), X(s)\rangle\ds\\
                                    &=& \langle R(k),X(t_k)\rangle +  V(t_k^+,X(t_k^+)).
    \end{array}
  \end{equation}
  This implies that $\langle \bar P(t_k^+),X(t_k^+)\rangle-\langle \bar P(t_k),X(t_k)\rangle+ \langle R(k),X(t_k)\rangle=0$. However, we also have that
\begin{equation}
  \begin{array}{rcl}
    \langle \bar P(t_k^+),X(t_k^+)\rangle-\langle \bar P(t_k),X(t_k)\rangle+ \langle R(k),X(t_k)\rangle &=&\langle \bar P(t_k^+),\mathcal{D}_k(X(t_k))\rangle-\langle \bar P(t_k),X(t_k)\rangle+ \langle R(k),X(t_k)\rangle\\
                                                                                                                                                                    &=&\langle \mathcal{D}_k^*(\bar P(t_k^+)),X(t_k)\rangle+\langle -\bar P(t_k)+R(k),X(t_k)\rangle\\
                                                                                                                                                                    &=&\langle \mathcal{D}_k^*(\bar P(t_k^+))-\bar P(t_k)+R(k),X(t_k)\rangle\\
                                                                                                                                                                    &=&0,
  \end{array}
\end{equation}
which shows that $\bar P$ in \eqref{eq:explicitP(t)} also verifies the second inequality in \eqref{eq:SecondOrderMomentImpulsiveLyapunov1} with $P=\bar P$. This proves the desired result.
\end{proof}}

\subsection{Relaxed stability results}

\black{The above results can be relaxed in two possible ways. The first one, called persistent flowing relaxation, relaxes the strict decrease of the Lyapunov function at jumps:
\begin{theorem}[Persistent flowing]\label{th:generalImpulsiveSystems:jdksjdlasdjakljdlsa:persflow}\label{th:Main:Imp:persflow}
    Assume that $\mathcal{D}_k(X)\preceq X$ for all $k\ge0$ and all $X\in\psd^n$. Then, the following statements are equivalent:
  \begin{enumerate}[(a)]
    %
    \item The matrix-valued linear dynamical system \eqref{eq:matdiffeqHybridLTV} is globally uniformly exponentially stable hybrid rate $(\alpha,1)$ for some $\alpha>0$.
    %
    \item There exist a piecewise differentiable matrix-valued function $P:\mathbb{R}_{\ge0}\mapsto\mathbb{S}^n_{\succ0}$ and a continuous matrix-valued function $Q:\mathbb{R}_{\ge0}\mapsto\mathbb{S}_{\succ0}^n$ such that
    \begin{equation}
    \begin{array}{rcccl}
      \alpha_1I&\preceq& P(t)&\preceq &\alpha_2I,\\
      \beta_1I&\preceq &Q(t)&\preceq &\beta_2I,
    \end{array}
    \end{equation}
such that
    \begin{equation}
      \begin{array}{rclcl}
         \dot{P}(t)+\mathcal{C}_t(P(t))&=&-Q(t),&& t\ne t_k,\ t\ge t^0,\ k\ge1\\
        \mathcal{D}_k(P(t_k^+))-P(t_k)&=&0,&& t_k>t^0,\ k\ge1
      \end{array}
    \end{equation}
     hold for some scalars $\alpha_1,\alpha_2,\beta_1,\beta_2>0$.
     \item There exists a piecewise differentiable matrix-valued function $P:\mathbb{R}_{\ge0}\mapsto\mathbb{S}^n_{\succ0}$ such that
         \begin{equation}
      \alpha_1I\preceq P(t)\preceq \alpha_2I
    \end{equation}
    and
    \begin{equation}
      \begin{array}{rclcl}
         \dot{P}(t)+\mathcal{C}_t(P(t))&\preceq&-\alpha_3I,&& t\ne t_k,\ t\ge t^0,\ k\ge1\\
        \mathcal{D}_k(P(t_k^+))-P(t_k)&\preceq&0,&& t_k>t^0,\ k\ge1
      \end{array}
    \end{equation}
    hold for some scalars $\alpha_1,\alpha_2,\alpha_3>0$.
  \end{enumerate}
\end{theorem}
\begin{proof}
  The proof follows from the same lines as the proof of Theorem \ref{th:generalImpulsiveSystems:jdksjdlasdjakljdlsa}. The only difference is in the proof of the implication that the statement (a) implies the statement (c). In the current proof, we consider a continuous matrix-valued function $Q$ and define the matrix-valued function $\bar{P}(t)$ as
  \begin{equation}
    \langle\bar{P}(t),X(t)\rangle:=\int_t^\infty\langle Q(s),X(s)\rangle\ds.
  \end{equation}
  As a result, the expression \eqref{eq:kds;kdksl;kad;ls;ddsdsddsdsk} becomes
  \begin{equation}
  -\langle Q(t),X(t)\rangle+\sum_{i=k}^\infty\langle Q(t_i),X(t_{i})-X(t_{i}^+)\rangle  \ge -c \langle\bar{P}(t),X(t)\rangle.
\end{equation}
Under the assumption that $\mathcal{D}_i(X)\preceq X$, we then get that
\begin{equation}
  -\langle Q(t),X(t)\rangle\ge -c \langle\bar{P}(t),X(t)\rangle
\end{equation}
which shows, in turn, that $\bar{P}(t)\succeq \beta_1/cI$. The rest of the proof follows from the same lines.
\end{proof}}

 The second one, called persistent jumping relaxation, relaxes the strict decrease of the Lyapunov function along the flow of the system:
\black{\begin{theorem}[Persistent jumping]\label{th:generalImpulsiveSystems:jdksjdlasdjakljdlsa:persjump}\label{th:Main:Imp:persjump}
    Assume that $\mathcal{C}_t(X)\preceq0$ for all $t\ge 0$ and all $X\in\psd^n$. Then, the following statements are equivalent:
  \begin{enumerate}[(a)]
    %
    \item The matrix-valued linear dynamical system \eqref{eq:matdiffeqHybridLTV} is globally uniformly exponentially stable with hybrid rate $(0,\rho)$ for some $\rho\in(0,1)$.
    %
    \item There exist a piecewise differentiable matrix-valued function $P:\mathbb{R}_{\ge0}\mapsto\mathbb{S}^n_{\succ0}$ and a matrix-valued function $R:\mathbb{Z}_{\ge0}\mapsto\mathbb{S}_{\succ0}^n$ such that
    \begin{equation}
    \begin{array}{rcccl}
      \alpha_1I&\preceq& P(t)&\preceq &\alpha_2I,\\
      \gamma_1I&\preceq &R(k)&\preceq &\gamma_2I,
    \end{array}
    \end{equation}
    and
    \begin{equation}
      \begin{array}{rclcl}
         \dot{P}(t)+\mathcal{C}_t(P(t))&=&0, && t\ne t_k,\ t\ge t^0,\ k\ge1\\
        \mathcal{D}_k(P(t_k^+))-P(t_k)&=&-R(k),&& t_k>t^0,\ k\ge1
      \end{array}
    \end{equation}
    hold for some $\alpha_1,\alpha_2,\gamma_1,\gamma_2>0$.
     \item There exists a piecewise differentiable matrix-valued function $P:\mathbb{R}_{\ge0}\mapsto\mathbb{S}^n_{\succ0}$ such that
     \begin{equation}
      \alpha_1I \preceq P(t) \preceq  \alpha_2I
    \end{equation}
    and
    \begin{equation}
      \begin{array}{rclcl}
         \dot{P}(t)+\mathcal{C}_t(P(t))&\preceq&0,&& t\ne t_k,\ t\ge t^0,\ k\ge1\\
        \mathcal{D}_k(P(t_k^+))-P(t_k)&\preceq&-\alpha_3I,&& t=t_k, t\ge t^0
      \end{array}
    \end{equation}
    hold for some $\alpha_1,\alpha_2,\alpha_3>0$.
  \end{enumerate}
\end{theorem}
\begin{proof}
    The proof follows from the same lines as the proof of Theorem \ref{th:generalImpulsiveSystems:jdksjdlasdjakljdlsa}. The only difference is in the proof of the implication that the statement (a) implies the statement (c). In the current proof, we define the matrix-valued function $\bar{P}(t)$ as
  \begin{equation}
     \langle \bar P(t),X(t)\rangle:=\sum_{t\le t_i}\langle R(i),X(t_i)\rangle
  \end{equation}
   Clearly, we have that $\langle \bar P(t),X(t)\rangle\ge\langle R(s_t),X(s_t)\rangle$ where $s_t:=\min\{i: t\le t_i\}$. Using now the assumption that $\mathcal{C}_t(X)\preceq0$ for all $t\ge 0$, we get that $X(s_t)\preceq X(t)$, which implies that
   \begin{equation}
      \langle \bar P(t),X(t)\rangle\ge\langle R(s_t),X(t)\rangle
   \end{equation}
   and that $\bar P(t)\succeq \gamma_1I$. This proves the lower bound for $\bar P$(t). The rest of the proof follows from the same lines as the proof of Theorem \ref{th:generalImpulsiveSystems:jdksjdlasdjakljdlsa}.
\end{proof}}

\subsection{Stabilization}


\blue{Let us consider here the perturbation of the system \eqref{eq:matdiffeqHybridLTV} given by
\begin{equation}\label{eq:matdiffeqHybridLTV:stabz}
\begin{array}{rcl}
  \dot{X}(t)&=&(A_0(t)+B_0(t)K_0(t)) X(t)+X(t)(A_0(t)+B_0(t)K_0(t))^\T\\
  && +\sum_{i=1}^N(A_i(t)+B_i(t)K_i(t)) X(t)(A_i(t)+B_i(t)K_i(t))^\T+\mu X(t),t\ne t_k,t\ge t^0\\
  X(t^+)&=&\sum_{i=0}^N(J_i(k)+E_i(k)G_i(k)) X(t)(J_i(k)+E_i(k)G_i(k))^\T,\ t=t_k, t\ge t^0\\
  X(t^0)&=&X_0\in\psd^n
\end{array}
\end{equation}
where the matrix-valued functions $A_i:\mathbb{R}_{\ge0}\mapsto\mathbb{R}^{n\times n}$, $B_i:\mathbb{R}_{\ge0}\mapsto\mathbb{R}^{n\times m_i}$, $K_i:\mathbb{R}_{\ge0}\mapsto\mathbb{R}^{m_i\times n}$, $i=0,\ldots,N$, are assumed to be piecewise continuous and bounded, and $J_i:\mathbb{Z}_{\ge0}\mapsto\mathbb{R}^{n\times n}$, $E_i:\mathbb{Z}_{\ge0}\mapsto\mathbb{R}^{n\times m_i}$, $G_i:\mathbb{R}_{\ge0}\mapsto\mathbb{R}^{m_i\times n}$, $i=0,\ldots,N$, are bounded. We then have the following result:
\begin{theorem}\label{th:generalImpulsiveSystems:jdksjdlasdjakljdlsa:stabz}\label{th:Main:Imp:stabz}
    The following statements are equivalent:
  \begin{enumerate}[(a)]
    %
    \item The matrix-valued linear dynamical system \eqref{eq:matdiffeqHybridLTV:stabz} is globally uniformly exponentially stabilizable with hybrid rate $(\alpha,\rho)$ for some $\alpha>0$ and $\rho\in(0,1)$.
    %
    \item  There exist a piecewise differentiable matrix-valued function $Q:\mathbb{R}_{\ge0}\mapsto\mathbb{S}^n_{\succ0}$, $\alpha_1I\preceq P(t)\preceq \alpha_2I$, such that
    \begin{equation}\label{eq:SecondOrderMomentImpulsiveLyapunov1:stabz}
      \begin{array}{rcl}
     \mathscr{N}(t)^{\T}\begin{bmatrix}
     -\dot{Q}(t)+\He[A_0(t)Q(t)]+\mu(t) Q(t) & \row_{i=1}^N[Q(t)A_i(t)^{\T}]\\
         \star & -I_N\otimes Q(t)
      \end{bmatrix}\mathscr{N}(t)&\preceq&-\alpha_3I,\ t\ne t_k,\ t\ge t^0,\ k\ge1\\
        \begin{bmatrix}
          I & 0\\
          0 & \mathscr{M}(k)^{\T}
        \end{bmatrix}\begin{bmatrix}
          -Q(t_k) & \row_{i=0}^N[Q(t_k)J_i(k)^{\T}]\\
          \star & -I_N\otimes Q(t_k^+)
        \end{bmatrix}  \begin{bmatrix}
          I & 0\\
          0 & \mathscr{M}(k)
        \end{bmatrix}&\preceq&-\alpha_4I,\ t_k>t^0,\ k\ge1
      \end{array}
    \end{equation}
    hold for some scalars $\alpha_1,\alpha_2,\alpha_3,\alpha_4>0$ where $\mathscr{N}(t)^{\T}$ is a basis of the left null-space of $\textstyle\diag_{i=0}^N(B_i(t))$ and $\mathscr{M}(k)^{\T}$ is a basis of the left null-space of $\textstyle\diag_{i=0}^N(E_i(k))$.
     \item There exist a piecewise differentiable matrix-valued function $Q:\mathbb{R}_{\ge0}\mapsto\mathbb{S}^n_{\succ0}$, $\alpha_1I\preceq P(t)\preceq \alpha_2I$, matrix valued functions $U_i:\mathbb{R}_{\ge0}\mapsto\mathbb{R}^{m_i\times n}$ and $V_i:\mathbb{Z}_{\ge0}\mapsto\mathbb{R}^{m_i\times n}$. $i=0,\ldots,N$,  such that
    \begin{equation}\label{eq:SecondOrderMomentImpulsiveLyapunov2:stabz}
      \begin{array}{rcl}
      \begin{bmatrix}
     -\dot{Q}(t)+\He[\Psi_0(t)]+\mu(t) Q(t) & \row_{i=1}^N[\Psi_i(t)^{\T}]\\
         \star & -I_N\otimes Q(t)
      \end{bmatrix}&\preceq&-\alpha_3I,\ t\ne t_k,\ t\ge t^0,\ k\ge1\\
        \begin{bmatrix}
          -Q(t_k) & \row_{i=0}^N[(\Phi_i(k))^{\T}]\\
          \star & -I_N\otimes Q(t_k^+)
        \end{bmatrix}&\preceq&-\alpha_4I,\ t_k>t^0,\ k\ge1
      \end{array}
    \end{equation}
    hold for some scalars $\alpha_1,\alpha_2,\alpha_3,\alpha_4>0$ where $\Psi_i(t):=A_i(t)Q(t)+ B_i(t)U_i(t)$ and $\Phi_i(k):=J_i(k)Q(t_k)+E_i(k)V_i(k)$.
  \end{enumerate}
\end{theorem}
\begin{proof}
The proof consists of a mixture of the proofs of Theorem \ref{th:CT:stabz} and Theorem \ref{th:DT:stabz}. It is thus omitted for brevity.
\end{proof}}

\subsection{Applications}

We illustrate in this section that several results from the literature can be recovered and/or extended using the proposed approach. In particular, we cover the cases of stochastic impulsive systems with deterministic impulse times \citep{Michel:08,Goebel:12,Briat:15i}, stochastic switched system with deterministic switching times \citep{Michel:08,Goebel:12,Briat:14f,Shaked:14,Briat:15i}, stochastic sampled-data systems \citep{Michel:08,Goebel:12,Briat:15i}, and impulsive Markov jump linear systems \citep{Souza:21}. This covers a broad range of systems of the literature.

\subsubsection{A class of stochastic LTV systems with deterministic impulses}

\black{We consider here the following class of stochastic LTV systems with deterministic jumps as
\begin{equation}\label{eq:syst:Impulsive_general_LTV}
  \begin{array}{rcl}
    \d x(t)&=&A_0(t)x(t)\dt+\sum_{i=1}^{N}A_i(t)x(t)\d W_i(t),\ t\ne t_k,\ k\ge1\\
    x(t_k^+)&=&J_0(k)x(t_k)+\sum_{i=1}^{N}J_i(k)x(t_k)\nu_i(k),\ k\ge1,\ t_k\ge t^0\\
    x(t^0)&=&x_0
  \end{array}
\end{equation}
where $x,x_0\in\mathbb{R}^n$ is the state of the system and the initial condition, $t^0$ is the initial time. The vector $W(t):=(W_1(t),\ldots,W_{N}(t))$ is a vector of $N$ independent Wiener processes which are also independent of the state $x(t)$ whereas the vector $\nu(k):=(\nu_1(k),\ldots,\nu_{N}(k))$ contains
$N$ independent random processes which are independent of the state $x(t_k)$ and which have zero mean and unit variance. The matrix-valued functions of the system are assumed to satisfy the same assumptions as for the system \eqref{eq:matdiffeqHybridLTV}. The sequence of impulse time ${t_k}_{\ge1}$ is assumed to be increasing and to grow unboundedly. To this system, we can associate a filtered probability space $(\Omega,\mathcal{F},\mathcal{F}_{t,k},\P)$ with hybrid filtration \citep{Teel:14} which satisfies the usual conditions.

The following notion of stability for the system \eqref{eq:syst:Impulsive_general_LTV} is considered here:
\begin{definition}\label{def:expstab}
  The system \eqref{eq:syst:Impulsive_general_LTV} is said to be uniformly exponentially mean-square stable with hybrid rate $(\alpha,\rho)$ if there exist an $M>0$, a $\rho\in(0,1)$ and an $\alpha>0$ such that we have
  \begin{equation}
    \E[||x(t)||_2^2]\le M\rho^{\kappa(t,t_0)}e^{-\alpha(t-t^0)}\E[||x_0||_2^2]
  \end{equation}
  for all $t\ge t^0$ and all $t^0\ge0$ where $\kappa(t,s)$ denotes the number of jumps in the interval $[s, t]$.
\end{definition}

We then have the following result which is an extension of the results in \citep{Briat:15i}:
\begin{theorem}\label{th:generalImpulsiveSystems}
  The following statements are equivalent:
  \begin{enumerate}[(a)]
    \item The system \eqref{eq:syst:Impulsive_general_LTV} is uniformly exponentially mean-square stable with hybrid rate $(\alpha,\rho)$ for some $\alpha>0$ and $\rho\in(0,1)$.
    %
    \item There exists a piecewise differentiable matrix-valued function $P:\mathbb{R}_{\ge0}\mapsto\mathbb{S}^n_{\succ0}$, $ \alpha_1I\preceq P(t)\preceq \alpha_2I$, such that
    \begin{equation}
      \begin{array}{rclcl}
        \dot{P}(t)+P(t)A_0(t)+A_0(t)^{\T}P(t)+\sum_{i=1}^NA_i(t)^{\T}P(t)A_i(t)&\preceq&-\alpha_3I,&& t\ne t_k,\ t\ge t^0,\ k\ge1\\
        \sum_{i=1}^NJ_i(k)^{\T}P(t_k^+)J_i(k)-P(t_k)&=&-\alpha_3I,&& t_k>t^0,\ k\ge1
      \end{array}
    \end{equation}
    hold for some scalars $\alpha_1,\alpha_2,\alpha_3>0$.
  \end{enumerate}
\end{theorem}
\begin{proof}
  Observing that the \blue{seconder-order moment} $X(t)=\E[x(t)x(t)^{\T}]$ of the state of the system \eqref{eq:syst:Impulsive_general_LTV} is governed by the system
    \begin{equation}
      \begin{array}{rcl}
        \dot{X}(t)&=&A_0(t)X(t)+X(t)A_0(t)^{\T}+\sum_{i=1}^{N}A_{i}(t)X(t)A_{i}^{\T},\ t\ne t_k,\ k\ge1,\ t\ge t^0,\\
        X(t_k^+)&=&\sum_{i=0}^{N}J_i(k)X(t_k)J_i(k)^{\T},\ k\ge1, t_k\ge t^0,\\
        X(t_0^+)&=&X(t_0)=X(0),
      \end{array}
    \end{equation}
    then an immediate application of Theorem \ref{th:generalImpulsiveSystems:jdksjdlasdjakljdlsa} yields the result.
\end{proof}}

\blue{Using the above very general result, we can retrieve the results obtained in \cite{Briat:15i}:
\begin{theorem}[Constant dwell-time]
    The following statements are equivalent:
  \begin{enumerate}[(a)]
    \item The LTI version of the system \eqref{eq:syst:Impulsive_general_LTV} is uniformly exponentially mean-square stable with hybrid rate $(\alpha,\rho)$ for some $\alpha>0$ and $\rho\in(0,1)$ under constant dwell-time $\bar T$.
    \item There exists a piecewise differentiable matrix-valued function $P:[0,\bar T]\mapsto\mathbb{S}^n_{\succ0}$ such that
    \begin{equation}
      \begin{array}{rcl}
        \dot{P}(\tau)+P(\tau)A_0+A_0^{\T}P(\tau)+\sum_{i=1}^NA_i^{\T}P(\tau)A_i&\preceq&-\alpha I,\  \tau\in[0,\bar T]\\
        \sum_{i=1}^NJ_i^{\T}P(0)J_i-P(\bar T)&\preceq&-\alpha I
      \end{array}
    \end{equation}
    hold for some scalars $\alpha>0$.
    \end{enumerate}
\end{theorem}
\begin{proof}
  Since the impulses arrive periodically, it is necessary and sufficient to consider a $\bar T$-periodic matrix-valued function $P(t_k+\tau)=P(\tau)$ for all $k\ge0$. Note that we have $P(t_k^+)=P(0)$. The result then follows.
\end{proof}

We can also retrieve this result which can be found in \cite{Briat:15i}:
\begin{theorem}[Minimum dwell-time]
    The following statements are equivalent:
  \begin{enumerate}[(a)]
    \item The LTI version of the system \eqref{eq:syst:Impulsive_general_LTV} is uniformly exponentially mean-square stable with hybrid rate $(\alpha,\rho)$ for some $\alpha>0$ and $\rho\in(0,1)$ under minimum dwell-time $\bar T$.
    \item There exists a piecewise differentiable matrix-valued function $P:[0,\bar T]\mapsto\mathbb{S}^n_{\succ0}$ such that
    \begin{equation}
      \begin{array}{rcl}
        \dot{P}(\tau)+P(\tau)A_0+A_0^{\T}P(\tau)\sum_{i=1}^NA_i^{\T}P(\tau)A_i&\preceq&-\alpha I,\  \tau\in[0,\bar T]\\
        P(\bar T)A_0+A_0^{\T} P(\bar T)+\sum_{i=1}^NA_i^{\T}P(\bar T)A_i&\preceq&-\alpha I,\  \tau\in[0,\bar T]\\
        \sum_{i=1}^NJ_i^{\T}P(0)J_i-P(\bar T)&\preceq&-\alpha I
      \end{array}
    \end{equation}
    hold for some scalars $\alpha>0$.
    \end{enumerate}
\end{theorem}
\begin{proof}
  In this case, we just need to consider
  \begin{equation}
    P(t_k+\tau)=\left\{\begin{array}{rcl}
      P(\tau)&& \textnormal{if }\tau\in[0,\bar T]\\
      P(\bar T)&& \textnormal{if }\tau\in[\bar T,T_k]
    \end{array}\right.
  \end{equation}
  and the result follows. Note that this matrix may fail to be differentiable at $\tau=\bar T$. However, this is not a problem as it has been shown in \cite{Holicki:19} that if there exists a matrix-valued function $P$ which is differentiable on $(0,\bar T)$, then there exists a $P$ such that $\dot{P}(\bar T)=0$ that sastifies the conditions of the result. This proves the result.
\end{proof}}

\subsubsection{A class of stochastic LTV switched systems with deterministic switchings}

\black{We illustrate in this section that the results obtained in this paper allow to retrieve and extend those developed in \citep{Shaked:14,Briat:14f}. To this aim, let us consider here the following class of LTV stochastic switched systems
\begin{equation}\label{eq:switched:det}
  \begin{array}{rcl}
    \d x(t)&=&A_{\sigma(t),0}(t)x(t)\dt+\sum_{i=1}^{N}A_{\sigma(t),i}(t)x(t)\d W_i(t)\\
    x(t_0)&=&x_0\\
    \sigma(t_0)&=&\sigma_0\\
  \end{array}
\end{equation}
where $x,x_0\in\mathbb{R}^n$ are the state of the system and the initial condition, and $t_0=0$ is the initial time. The Wiener processes $W_1(t),\ldots,W_{N}(t)$ are assumed to be independent of each other and of the state $x(t)$ and the switching signal $\sigma(t)$ of the system. The switching signal $\sigma:\mathbb{R}_{\ge0}\mapsto\{1,\ldots,M\}$, with initial value $\sigma_0$, is assumed to be deterministic and piecewise constant. The switching signal is assumed to change values at time $t_k\ge0$, $k\ge1$, and this sequence is assumed to be increasing and to grow unboundedly. The dwell-time $T_k$ is defined as $T_k:=t_{k+1}-t_k$, $k\ge0$. The notion of mean-square stability in Definition \ref{def:expstab} can be straightforwardly adapted to the system above. This leads us to the following result:
\begin{theorem}[Fixed dwell-time]\label{cor:switched:fixedDT}
  Assume that there exist differentiable matrix-valued functions $P_i:\mathbb{R}_{\ge0}\mapsto\pd^n$, $\alpha_1I\preceq P_i(\cdot)\preceq\alpha_2I$, $i=1,\ldots,M$, such that
    \begin{equation}
      \begin{array}{rcl}
         \dot{P}_i(t)+\He[P_i (\tau)A_{0,i}(t)]+\sum_{j=1}^NA_{i,j}(t)^{\T}P_i (t)A_{i,j}(t)&\preceq&-\alpha_3I,\ i =1,\ldots,M,\ t\ne t_k,\ t\ge t^0,\ k\ge1\\
         P_{i}(t_k^+)-P_{j}(t_k)&\preceq&0,\ i,j=1,\ldots,M,\ i\ne j,\ k\ge1
      \end{array}
    \end{equation}
  hold for some $\alpha_1,\alpha_2,\alpha_3>0$.

  Then, the system \eqref{eq:switched:det} is mean-square exponentially stable under the fixed dwell-time sequence $\{T_k\}_{k\ge0}$.
\end{theorem}
\begin{proof}
The dynamics of the \blue{seconder-order moment}  matrix $X(t)=\E[x(t)x(t)^{\T}]$ is given by
\begin{equation}
\begin{array}{rcl}
  \dot{X}(t)&=&A_{\sigma(t),0}(t)X(t)+A_{\sigma(t),0}(t)^{\T}X(t)+\sum_{j=1}^{N}A_{\sigma(t),\ell}^{\T}X(t)A_{\sigma(t),\ell}\\
  X(t_k^+)&=&X(t_k)
\end{array}
\end{equation}
Since, $\mathcal{D}_k(X)-X\preceq0$, then we can apply Theorem \ref{th:generalImpulsiveSystems:jdksjdlasdjakljdlsa:persflow} to yield the result.
\end{proof}

When the sequence of dwell-times satisfies a certain condition, we can specialize the above result to obtain more tractable stability conditions. The next result applies to the case when the constant dwell-time case, that is, when $T_k=\bar T$, $k\ge0$, for some constant dwell-time value $\bar T>0$ and can be seen as an extension of the results in \cite{Allerhand:11,Shaked:14,Briat:14f}:
\begin{corollary}[Constant dwell-time]\label{cor:switched:cstDT}
  Assume that there exist differentiable matrix-valued functions $P_i:\mathbb{R}_{\ge0}\mapsto\pd^n$, $i=1,\ldots,M$, such that
    \begin{equation}
      \begin{array}{rclcl}
         \dot{P}_{i}(\tau)+A_{i,0}^{\T}P_{i}(\tau)+P_{i}(\tau)\bar{A}_{i,0}+\sum_{j=1}^N A_{i,j}^{\T}P_{i}(\tau)A_{i,j}&\preceq&-\alpha I,&& \tau\in[0,\bar T],\ i=1,\ldots,M,\\
        P_{i}(0)-P_{j}(\bar T)&\preceq&0,&& i,j=1,\ldots,M,
      \end{array}
    \end{equation}
  hold for some  $\alpha>0$.

  Then, the LTI version of the system \eqref{eq:switched:det} is mean-square exponentially stable under eventual constant dwell-time $\bar T$.
\end{corollary}
\begin{proof}
  The proof simply follows from picking $P(t_k+\tau)=P(\tau)$ with $P(0)=P(t_k^+)$.
\end{proof}

The following result addresses the case where the dwell-time sequence satisfies $T_k\ge \bar T$, $k\ge1$, for some minimum dwell-time value $\bar T$ and can also be seen as a generalization of the results in \cite{Allerhand:11,Shaked:14,Briat:14f}:
\begin{corollary}[Minimum dwell-time]\label{cor:switched:minDT}
  Assume that there exist differentiable matrix-valued functions $P_i:\mathbb{R}_{\ge0}\mapsto\pd^n$, $i=1,\ldots,M$, such that
    \begin{equation}
      \begin{array}{rclcl}
          A_{i,0}^{\T}P_{i}(\bar T)+P_{i}(\bar T)\bar{A}_{i,0}+\sum_{j=1}^N A_{i,j}^{\T}P_{i}(\bar T)A_{i,j}&\preceq&-\alpha I,&& i=1,\ldots,M,\\
         \dot{P}_{i}(\tau)+A_{i,0}^{\T}P_{i}(\tau)+P_{i}(\tau)\bar{A}_{i,0}+\sum_{j=1}^N A_{i,j}^{\T}P_{i}(\tau)A_{i,j}&\preceq&-\alpha I,&& \tau\in[0,\bar T],\ i=1,\ldots,M,\\
        P_{i}(0)-P_{j}(\bar T)&\preceq&0,&& i,j=1,\ldots,M,
      \end{array}
    \end{equation}
  hold for some  $\alpha>0$.

  Then, the LTI version of the system \eqref{eq:switched:det} is mean-square exponentially stable under minimum dwell-time $\bar T$.
\end{corollary}
\begin{proof}
The proof follows from considering all the matrix valued functions $P_i(\tau)$ such that $P_i(\tau)=P_i(\bar T)$ for all $\tau\ge\bar T$. Even though those functions are not differentiable at $\tau=\bar T$, then it is possible to shown that if the conditions in the result hold, then there exist differentiable functions that satisfy the conditions of the result; see e.g. \citep{Holicki:19}.
\end{proof}

The following result addresses the case where the dwell-time sequence satisfies $T_k\in[\Tmin^{\sigma(t_k^+)},\Tmax^{\sigma(t_k^+)}]$, $k\ge1$, for some range dwell-time value $[\Tmin^1,\Tmax^1]\times\ldots\times[\Tmin^M,\Tmax^M]$, and can be seen as an extension of the results in \cite{Briat:14f}:
\begin{corollary}[Mode-dependent range dwell-time]\label{cor:switched:rangeDT}
  Assume that there exist differentiable matrix-valued functions $P_i:\mathbb{R}_{\ge0}\mapsto\pd^n$, $i=1,\ldots,M$, such that
    \begin{equation}
      \begin{array}{rclcl}
         \dot{P}_i (\tau)+A_{0,i }^{\T}P_i (\tau)+P_i (\tau)A_{0,i }+\sum_{j=1}^NA_{i,j}^{\T}P_i (\tau)A_{i,j}&\preceq&-\alpha I,&& i =1,\ldots,M,\ \tau\in[0,\Tmax^i],\\
         P_{i}(0)-P_{j}(\theta)&\preceq&0,&& i,j=1,\ldots,M,\ i\ne j,\ \theta\in[\Tmin^j,\Tmax^j],
      \end{array}
    \end{equation}
  hold for some $\alpha>0$.

  Then, the LTI version of the system \eqref{eq:switched:det} is mean-square exponentially stable under the mode-dependent range dwell-time $[\Tmin^1,\Tmax^1]\times\ldots\times[\Tmin^M,\Tmax^M]$.
\end{corollary}
\begin{proof}
The proof simply consists of considering $T_k^i$ to lie within the interval $[\Tmin^i,\Tmax^i]$, $i=1,\ldots,M$, in the conditions.
\end{proof}}


\subsubsection{Sampled-data stochastic LTV systems}


\blue{We illustrate in this section that the results obtained in this paper allow to retrieve and extend those developed in \citep{Briat:15i}. To this aim, let us consider again the stochastic LTV sampled-data system \eqref{eq:syst:Impulsive_SD1}-\eqref{eq:syst:Impulsive_SD2} and assume now that the sampling times are deterministic. This leads to the following result that can be seen as an extension of the results in \citep{Briat:15i}:
\begin{theorem}\label{cor:LTV:SD}
  The following statements are equivalent:
  \begin{enumerate}[(a)]
    \item The LTV stochastic sampled-data system \eqref{eq:syst:Impulsive_SD1}-\eqref{eq:syst:Impulsive_SD2} is uniformly mean-square exponentially stable.
    \item There exist a differentiable matrix-valued function $P:\mathbb{R}_{\ge0}\mapsto\pd^n$, $\alpha_1I\preceq P(\cdot)\preceq \alpha_2I$ and two matrix-valued functions $K_1:\mathbb{Z}_{\ge0}\mapsto\mathbb{R}^{m\times n}$ and $K_2:\mathbb{Z}_{\ge0}\mapsto\mathbb{R}^{m\times m}$ such that
    \begin{equation}
  \dot{P}(t)+\bar{A}_0(t)^{\T}  P(t)+  P(t)\bar{A}_0(t)+\sum_{i=1}^N\bar{A}_i(t)^{\T}  P(t)\bar{A}_i(t)\preceq- \alpha_3I,\ t\ne t_k,\ t\ge t^0
\end{equation}
and
\begin{equation}
  \bar{J}(k)^{\T}  P(t_k^+)\bar{J}(k)-  P(t_k)\preceq- \alpha_3I,\ k\ge1, t_k> t^0
\end{equation}
hold for some $\alpha_1,\alpha_2,\alpha_3>0$  where
\begin{equation}
  \bar{A}_i(t):=\begin{bmatrix}
    A_i(t) & B_i(t)\\
    0 & 0
  \end{bmatrix},\ \textnormal{and }\bar J(k):=\begin{bmatrix}
    I & 0\\
    K_1(k) & K_2(k).
  \end{bmatrix}
\end{equation}
\item There exist a differentiable matrix-valued function $Q:\mathbb{R}_{\ge0}\mapsto\pd^n$, $\alpha_1I\preceq Q(\cdot)\preceq \alpha_2I$ and a matrix-valued function $U:\mathbb{Z}_{\ge0}\mapsto\mathbb{R}^{(m+n)\times n}$ such that
    \begin{equation}
    \begin{bmatrix}
         -\dot{Q}(t)+Q(t)\bar{A}_0(t)^{\T} +  \bar{A}_0(t)Q(t) & \row_{i=1}^N[Q(t)\bar{A}_i(t)^{\T}]\\
         \star & -I_N\otimes Q(t)
    \end{bmatrix} \preceq- \alpha_3I,\ t\ne t_k,\ t\ge t^0
\end{equation}
and
\begin{equation}
\begin{bmatrix}
    -Q(t_k) & (\bar{J}_0Q(t_k)+BU(k))^{\T}\\
  \star & -Q(t_k^+)
\end{bmatrix}\preceq- \alpha_3I,\ k\ge1, t_k> t^0
  \end{equation}
  hold for some $\alpha_1,\alpha_2,\alpha_3>0$ and where
  \begin{equation}
  \bar{J}_0:=\begin{bmatrix}
    I & 0\\
    0 & 0
  \end{bmatrix},\ \textnormal{and }\bar B:=\begin{bmatrix}
    0\\
    I
  \end{bmatrix}.
\end{equation}
  Moreover, suitable gains are given by $\begin{bmatrix} K_1(k) & K_2(k)\end{bmatrix}=U(k)Q(t_k)^{-1}$.
  \end{enumerate}
\end{theorem}
\begin{proof}
  The  system \eqref{eq:syst:Impulsive_SD1}-\eqref{eq:syst:Impulsive_SD2} can be reformulated as the following impulsive system
\begin{equation}
  \d z(t)=\bar{A}_0(t)z(t)\dt+\sum_{i=1}^N\bar{A}_i(t)z(t)\d W_i(t)
\end{equation}
and
\begin{equation}
  z(t_k^+)=\bar Jz(t_k)
\end{equation}
where $A_i$, $i=1,\ldots,M$, and $\bar J(k)$ are defined in the result. The \blue{seconder-order moment} system associated with the above impulsive system is given by
\begin{equation}
  \dot{X}(t)=\bar{A}_0(t)X(t)+X(t)\bar{A}_0(t)^{\T}+\sum_{i=1}^N\bar{A}_i(t)X(t)\bar{A}_i(t)^{\T}
\end{equation}
and
\begin{equation}
  X(t_k^+)=\bar{J}X(t_k)\bar{J}^{\T}.
\end{equation}
Applying then Theorem \ref{th:generalImpulsiveSystems:jdksjdlasdjakljdlsa} yields the result.
\end{proof}

The following results can be seen as generalizations of those in \citep{Briat:13d,Briat:15i}:
\begin{corollary}[Periodic sampling]\label{cor:LTV:SD:cstDT}
    The following statements are equivalent:
    \begin{enumerate}[(a)]
    \item The LTI version of the sampled-data system \eqref{eq:syst:Impulsive_SD1}-\eqref{eq:syst:Impulsive_SD2} is mean-square exponentially stable under periodic sampling $\bar T$.
    %
\item There exist a differentiable matrix-valued function $Q:[0,\bar T]\mapsto\pd^n$ and a matrix $U\in\mathbb{R}^{(m+n)\times n}$ such that
    \begin{equation}
    \begin{bmatrix}
         -\dot{Q}(\tau)+Q(\tau)\bar{A}_0^{\T} +  \bar{A}_0Q(\tau) & \row_{i=1}^N[Q(\tau)\bar{A}_i^{\T}]\\
         \star & -I_N\otimes Q(\tau)
    \end{bmatrix} \preceq- \alpha I,\ \tau\in[0,\bar T]
\end{equation}
and
\begin{equation}
\begin{bmatrix}
    -Q(\bar T) & (\bar{J}_0Q(\bar T)+BU)^{\T}\\
  \star & -Q(0)
\end{bmatrix}\preceq- \alpha I,\ k\ge1, t_k> t^0
  \end{equation}
  hold  for some $\alpha>0$.where
  \begin{equation}
  \bar{J}_0:=\begin{bmatrix}
    I & 0\\
    0 & 0
  \end{bmatrix},\ \bar B:=\begin{bmatrix}
    0\\
    I
  \end{bmatrix}.
\end{equation}
  Moreover, suitable gains are given by $\begin{bmatrix} K_1 & K_2\end{bmatrix}=UQ(\bar T)^{-1}$.
  \end{enumerate}

\end{corollary}
\begin{proof}
  Due to periodicity of the impulse times period, we can choose $P(t_k+\tau)=P(\tau)$, $\tau\in(0,\bar T]$, in the conditions of Theorem \ref{cor:LTV:SD}, which proves the result.
\end{proof}

\begin{corollary}[Aperiodic sampling]\label{cor:LTV:SD:cstDT}
    The following statements are equivalent:
    \begin{enumerate}[(a)]
    \item The LTI version of the sampled-data system \eqref{eq:syst:Impulsive_SD1}-\eqref{eq:syst:Impulsive_SD2} is mean-square exponentially stable under aperiodic sampling $[\Tmin,\Tmax]$.
    %
\item There exist a differentiable matrix-valued function $Q:[0,\bar T]\mapsto\pd^n$ and a matrix $U:[0,\Tmax]\mapsto\mathbb{R}^{(m+n)\times n}$ such that
    \begin{equation}
    \begin{bmatrix}
         -\dot{Q}(\tau)+Q(\tau)\bar{A}_0^{\T} +  \bar{A}_0Q(\tau) & \row_{i=1}^N[Q(\tau)\bar{A}_i^{\T}]\\
         \star & -I_N\otimes Q(\tau)
    \end{bmatrix} \preceq- \alpha I,\ \tau\in[0,\Tmax]
\end{equation}
and
\begin{equation}
\begin{bmatrix}
    -Q(\theta) & (\bar{J}_0Q(\bar T)+BU(\theta))^{\T}\\
  \star & -Q(0)
\end{bmatrix}\preceq- \alpha I,\ \theta\in[\Tmin,\Tmax]
  \end{equation}
  hold  for some $\alpha>0$.where
  \begin{equation}
  \bar{J}_0:=\begin{bmatrix}
    I & 0\\
    0 & 0
  \end{bmatrix},\ \bar B:=\begin{bmatrix}
    0\\
    I
  \end{bmatrix}.
\end{equation}
  Moreover, suitable gains are given by $\begin{bmatrix} K_1(T_k) & K_2(T_k)\end{bmatrix}=UQ(T_k)^{-1}$.
  \end{enumerate}
\end{corollary}
\begin{proof}
  Due to periodicity of the impulse times period, we can choose $P(t_k+\tau)=P(\tau)$, $\tau\in(0,\Tmax]$, in the conditions of Theorem \ref{cor:LTV:SD}, which proves the result.
\end{proof}}

\subsubsection{A class of stochastic LTV impulsive systems with stochastic impulses and switching}

We consider here the following class of systems that extends the systems analyzed in \citep{Souza:21} to the stochastic and time-varying case:
\begin{equation}\label{eq:souza:1}
  \begin{array}{rcl}
    \dx(t)&=&A_{0,\sigma(t)}x(t)\dt+\sum_{i=1}^{N}A_{i,\sigma(t)}(t)x(t)\d W_i(t),\ \sigma(t)\in\{1,\ldots,M_c\},\\
    x(t_k^+)&=&J_{0,\sigma(t_k)}(k)x(t_k)+\sum_{i=1}^{N}J_{i,\sigma(t_k)}(k)\nu_i(k)x(t_k),\ \sigma(t_k)\in\{M_c+1,\ldots,M_c+M_d\}\\
    x(t^0)&=&x_0\\
    \sigma(t^0)&=&\sigma_0
  \end{array}
\end{equation}
where $M_c$ is the number of continuous modes and $M_d$ is the number of discrete modes. As for the other systems $x,x_0\in\mathbb{R}^n$ are the state of the system and the initial condition, and $t^0\ge0$ is the initial time. The Wiener processes $W_1(t),\ldots,W_{N}(t)$ are assumed to be independent of each other,  of the state $x(t)$, and the switching signal $\sigma(t)$ of the system. Similarly, the random sequences $\nu_1(k),\ldots,\nu_N(k)$ are independent of each other, of $x(t_k)$, and the switching signal $\sigma(t_k)$. The matrix-valued functions describing the system are also assumed to be uniformly bounded.  The switching signal $\sigma:\mathbb{R}_{\ge0}\mapsto\{1,\ldots,M\}$, with initial value $\sigma_0$, is assumed to be stochastically varying and piecewise constant. We also have that for $i=1,\ldots,M_c$:
\begin{equation}\label{eq:souza:2}
  p_{ij}(h)=\mathbb{P}(\sigma(t+h)=j|\sigma(t)=i)=\left\{\begin{array}{rcl}
    \lambda_{ij}h+o(h),&&i\ne j\\
    1+\lambda_{ii}h+o(h),&&i=j
  \end{array}\right.
\end{equation}
where $\Lambda=[\lambda_{ij}]\in\mathbb{R}^{M_c\times(M_c+M_d)}$, $\lambda_{ij}\ge0$ for all $i\ne j$ and $\textstyle\sum_{j=1}^{M_c+M_d}\lambda_{ij}=0$ for all $i$. For $i=M_c+1,\ldots,M_c+M_d$, we have that
\begin{equation}\label{eq:souza:3}
  p_{ij}=\mathbb{P}(\sigma(t^+)=j|\sigma(t)=i)=\pi_{ij}
\end{equation}
where $\Pi=[\pi_{ij}]\in\mathbb{R}^{M_d\times(M_c+M_d)}$, $\pi_{ij}\ge0$ for all $i,j$, $\textstyle\sum_{j=1}^{M_c+M_d}\pi_{ij}=1$ for all $i$, and $\pi_{ij}=0$ for all $i,j\in\{M_c+1,\ldots,M_c+M_d\}$ (i.e. no transition between the discrete modes). We then have the following result that can be seen as an extension or a generalization of  \citep[Theorem 3.2]{Souza:21} to the stochastic and time-varying case:
\begin{theorem}\label{th:souza}
The following statements are equivalent:
\begin{enumerate}[(a)]
  \item The system \eqref{eq:souza:1}, \eqref{eq:souza:2}, \eqref{eq:souza:3} is uniformly exponentially mean-square stable.
  \item There exist differentiable matrix-valued functions $P_i:\mathbb{R}_{\ge0}\mapsto\pd^n$, $\alpha_1I\preceq P_i(\cdot)\preceq\alpha_2I$, $i=1,\ldots,M_c+M_d$, such that
    \begin{equation}
         \dot{P}_i(t)+A_{0,i }(t)^{\T}P_i (\tau)+P_i (t)A_{0,i }(t)+\sum_{j=1}^NA_{i,j}(t)^{\T}P_i(t)A_{i,j}(t)+\sum_{j=1}^{M_c+M_d}\lambda_{ij}P_j(t)\preceq-\alpha_3I
    \end{equation}
    holds for all $i =1,\ldots,M_c$, $t\ne t_k$, and
        \begin{equation}
        J_\ell (k)^{\T}\left(\sum_{j=1}^{M_c}\pi_{\ell j}P_j(t_k^+)\right)J_\ell (k)-P_\ell (t_k)\preceq-\alpha_4I,
    \end{equation}
  holds for all $\ell =M_c+1,\ldots,M_c+M_d$, $t_k\ge t^0$ and for some $\alpha_1,\alpha_2,\alpha_3,\alpha_4>0$.
\end{enumerate}
\end{theorem}
\begin{proof}
The complete proof is omitted but can be seen as a combination of the proofs of Theorem \ref{th:generalImpulsiveSystems}, Theorem \ref{th:Jump:CT}, and Theorem \ref{th:Markov:DT}.
\end{proof}

\section{Concluding statements and future works}

\blue{A unified formulation for the analysis of a broad class of linear systems in terms of the analysis of matrix-valued differential equations has been introduced. The approach is shown to generalize and unify most of the results on the literature on linear systems including continuous-time, discrete-time, impulsive, switched and sampled-data systems, and certain systems with delays. The natural next step is the consideration of systems with inputs and the derivation of a dissipativity theory for such systems \citep{Willems:71,Brogliato:07,vanderSchaft:00}. This would lead, for instance, to the derivation of necessary and sufficient conditions characterizing their passivity, $L_2$-gain, etc.. The generalization of the approach to nonlinear matrix-valued monotone systems \citep{Angeli:03} would also be of great interest. Possible extensions on the stochastic side would be the consideration of Levy processes \citep{Applebaum:09} and of more general renewal processes along the same lines as \citep{Antunes:09,Antunes:09b,Antunes:10,Antunes:13,Lampersky:14}. While only the stabilization with respect to static state feedback has been considered here, the case of dynamic output feedback is still of great interest; see e.g. \cite{Scherer:90a,Scherer:97a,Geromel:19}. The design of observers for matrix-valued dynamical systems is also of great interest.}


\end{document}